\documentclass[11pt,reqno]{article}
\usepackage[utf8]{inputenc}
\usepackage{booktabs} 
\usepackage{array} 
\usepackage{paralist} 
\usepackage{verbatim} 

\usepackage{amsthm}
\usepackage[table,xcdraw]{xcolor}
\usepackage[titletoc,toc,title]{appendix}
\usepackage{latexsym, amsmath, amssymb, a4, epsfig, color}
\usepackage{blindtext}
\usepackage{graphicx}
\usepackage{subcaption}
\usepackage[utf8]{inputenc}
\usepackage[export]{adjustbox}
\usepackage{wrapfig}
\usepackage{apptools}
\usepackage{afterpage}

\usepackage{floatrow}

\usepackage{amssymb}
\usepackage{amsmath}
\usepackage[multiple]{footmisc}

\AtAppendix{\counterwithin{definition}{subsection}}
\AtAppendix{\counterwithin{theorem}{subsection}}

\graphicspath{}

\newtheorem{theorem}{Theorem}[section]
\newtheorem{cor}[theorem]{Corollary}
\newtheorem{lemma}[theorem]{Lemma}
\newtheorem{prop}[theorem]{Proposition}
\newtheorem{definition}{Definition}

\newtheorem{example}{Example}
\newtheorem{remark}{Remark}

\newcommand{\RR}{\mathbb{R}}
\newcommand{\CC}{\mathbb{C}}
\newcommand{\NN}{\mathbb{N}}
\newcommand{\ZZ}{\mathbb{Z}}
\newcommand{\TT}{\mathbb{T}}
\newcommand{\Om}{\Omega}
\newcommand{\ds}{\displaystyle}

\newcommand{\p}{\partial}
\newcommand{\pd}[2]{\frac {\p #1}{\p #2}}
\newcommand{\eqnref}[1]{(\ref {#1})}
\renewcommand{\qed}{\hfill $\Box$ \medskip}
\newcommand{\beq}{\begin{equation}}
\newcommand{\eeq}{\end{equation}}

\newcommand{\SingleOmega}{\mathcal{S}_{\partial\Omega}}
\newcommand{\DoubleOmega}{\mathcal{D}_{\partial\Omega}}
\newcommand{\KstarOmega}{\mathcal{K}_{\partial\Omega}^{*}}
\newcommand{\KOmega}{\mathcal{K}_{\partial\Omega}}
\newcommand{\Kcal}{\mathcal{K}}
\newcommand{\Scal}{\mathcal{S}}

\newcommand{\Kzeta}{K^{-1/2}}
\newcommand{\Keta}{K^{1/2}}

\newcommand{\Rtwo}{\mathbb{R}^2}
\newcommand{\LtwobdOmega}{L^2(\partial\Omega)}

\newcommand{\tzeta}{\widetilde{\zeta}}
\newcommand{\teta}{\widetilde{\eta}}
\newcommand{\ttheta}{\tilde{\theta}}

\newcommand{\la}{\langle}
\newcommand{\ra}{\rangle}

\numberwithin{equation}{section}
\numberwithin{figure}{section}

\begin{document}

\newcommand{\TheTitle}{A new series solution method for the transmission problem}
\newcommand{\TheAuthors}{Y. Jung and M. Lim}

\title{{\TheTitle}\thanks{{This work is supported by the Korean Ministry of Science, ICT and Future Planning through NRF grant No. 2016R1A2B4014530 (to Y.J. and M.L.).}}}
\author{
YoungHoon Jung\thanks{\footnotesize Department of Mathematical Sciences, Korea Advanced Institute of Science and Technology, Daejeon 305-701, Korea ({hapy1010@kaist.ac.kr},  {mklim@kaist.ac.kr}).} \and Mikyoung Lim\footnotemark[2]\ \thanks{\footnotesize Cooresponding author}}

\date{\today}
\maketitle
\begin{abstract}
In this paper, we propose a novel framework for the conductivity transmission problem in two dimensions with a simply connected inclusion of arbitrary shape. 
We construct a collection of harmonic basis functions, associated with the inclusion, based on complex geometric function theory.
It is well known that the solvability of the transmission problem can be established via the boundary integral formulation in which the Neumann-Poincar\'{e} (NP) operator is involved. 
The constructed basis leads to explicit series expansions for the related boundary integral operators. In particular, the NP operator becomes a doubly infinite, self-adjoint matrix operator, whose entry is given by the Grunsky coefficients corresponding to the inclusion shape. This matrix formulation provides us a simple numerical scheme to compute the transmission problem solution and, also, the spectrum of the NP operator for a smooth domain, by use of the finite section method.
The proposed geometric series solution method requires us to know the exterior conformal mapping associated with the inclusion. 
We derive an explicit boundary integral formula, with which the exterior conformal mapping can be numerically computed, so that one can apply the method for an inclusion of arbitrary shape.  We provide numerical examples to demonstrate the effectiveness of the proposed method. \\

Dans cet article, nous proposons un nouveau cadre pour le probl\`{e}me de conductivit\'{e} en deux dimensions avec une inclusion de conductivit\'{e} simplement connect\'{e}e et de forme arbitraire. 
Bas\'{e}e sur la th\'{e}orie de la fonction g\'{e}om\'{e}trique complexe, nous construisons une collection des fonctions harmoniques li\'{e}e \`{a} l'inclusion.  Le fait que la solvabilit\'{e} du probl\`{e}me de conductivit\'{e} avec une inclusion peut \^{e}tre \'{e}tablie par l'op\'{e}rateur Neumann-Poincar\'{e} est bien connu.  Avec les fonctions de base construites, l'op\'{e}rateur Neumann-Poincar\'{e} devient une matrice auto-adjointe en dimension infinie, dont l'entr\'{e}e est d\'{e}fini par les coefficients de Grunsky associ\'{e}s \`{a} la g\'{e}om\'{e}trie de l'inclusion. Sur la base de cette formulation matricielle, nous d\'{e}rivons un sch\'{e}ma num\'{e}rique simple pour calculer la solution du probl\`{e}me de transmission et, \'{e}galement, le spectre de l'op\'{e}rateur Neumann-Poincar\'{e} sur les domaines planaires de forme arbitraire. Nous d\'{e}rivons aussi une formule de l'int\'{e}grale au bord explicitement pour la transformation conforme de l'ext\'{e}rieur. Avec cette formule nous pouvons calculer la transformation conforme de l'ext\'{e}rieur des domaines planaires de forme arbitraire. Nous effectuons des exp\'{e}riences num\'{e}riques afin de d\'{e}montrer l'efficacit\'{e} de la m\'{e}thode.

\end{abstract}




\section{Introduction}

 The aim of this paper is to provide an analytical framework for the transmission problem in two dimensions that is applicable to an object of arbitrary shape. We let $\Omega$ be a simply connected bounded domain in $\mathbb{R}^2$ with a piecewise smooth boundary, possibly with corners, where the background is homogeneous with dielectric constant $\epsilon_m$ and $\Omega$ is occupied by a material of dielectric constant $\epsilon_c$. We consider the interface problem
\beq\label{cond_eqn0}
\begin{cases}
\ds\nabla\cdot\sigma\nabla u=0\quad&\mbox{in }\RR^2, \\
\ds u(x) - H(x)  =O({|x|^{-1}})\quad&\mbox{as } |x| \to \infty
\end{cases}
\end{equation}
with $
\sigma=\epsilon_c\chi_{\Omega}+\epsilon_m\chi_{\mathbb{R}^2\setminus \overline{\Omega}}$ for a given background field $H$. The symbol $\chi$ indicates the characteristic function. The solution $u$ should satisfy the transmission condition
$$u\big|^+=u\big|^-\quad\mbox{and }\quad \epsilon_m\pd{u}{\nu}\Big|^+=\epsilon_c\pd{u}{\nu}\Big|^-\qquad\mbox{a.e. on }\partial\Omega.$$
Here, $\nu$ is the outward unit normal vector on $\p\Om$ and the symbols $+$ and $-$ indicate the limit from the exterior and interior of $\Om$, respectively.
We may interpret the conductivity problem as the quasi-static formulation of electric fields or anti-plane elasticity. 
In recent years there has been increased interest in the analysis of the transmission problem in relation to applications in various areas such as inverse problems, invisibility cloaking, and nano-photonics \cite{Ammari:2004:RSI:book, Ciraci:2012:PUL,Milton:2006:CEA,Pendry:2006:CEF}.

A classical way to solve \eqnref{cond_eqn0} is to use the layer potential ansatz
\begin{equation}\label{eqn:layerpotentialansatz}
u(x) = H(x)+\SingleOmega[\varphi](x),	
\end{equation}
where $\SingleOmega$ indicates the single layer potential associated with the fundamental solution to the Laplacian and $\varphi$ involves the inversion of $\lambda I - \KstarOmega$, where $\lambda =\frac{\epsilon_c+\epsilon_m}{2(\epsilon_c-\epsilon_m)}$  and $\KstarOmega$ is the Neumann--Poincar\'{e} (NP) operator $\KstarOmega$ (see \cite{Escauriaza:1992:RTW, Escauriaza:1993:RPS,Kellogg:1929:FPT,Kenig:1994:HAT}). We reserve the mathematical details for the next section. The boundary integral equation can be numerically solved with high precision even for domains with corners \cite{Helsing:2013:SIE}.

In the present paper, we provide a new series solution method to the transmission problem for a domain of arbitrary shape.
When $\Omega$ has a simple shape such as a disk or an ellipse, there are globally defined orthogonal coordinates, namely the polar coordinates or the elliptic coordinates. In these coordinate systems the transmission problem \eqnref{cond_eqn0} can be solved by analytic series expansion; see for example \cite{Ammari:2019:SRN, Ando:2016:APR, Milton:2006:CEA}. In \cite{Ammari:2019:SRN}, an algebraic domain, which is the image of the unit disk under the complex mapping $w+\frac{a}{w^m}$ for some $m\in\mathbb{N},a\in\mathbb{R}$, was considered. For a domain of arbitrary shape, unlike in the case of a disk or an ellipse, an orthogonal coordinate system can be defined only locally and there is none known that is defined on the whole space $\mathbb{R}^2.$ This is the main obstacle when we seek to find the explicit series solution to \eqnref{cond_eqn0}.
 As far as we know, there has been no previous work that provides a series solution to the transmission problem for a domain of arbitrary shape.
The key idea of our work is to define a curvilinear orthogonal coordinate system only on the exterior region $\mathbb{R}\setminus{\overline{\Omega}}$ by using the exterior conformal mapping associated with $\Om$. We construct harmonic basis functions in the exterior region that decay at infinity by using the coordinates. We then adopt the Faber polynomials, first introduced by G. Faber in \cite{Faber:1903:PE}, as basis functions on the interior of $\Omega.$ It is worth mentioning that the Faber polynomials have been widely adopted in classical subjects of analysis such as univalent function theory \cite{Duren:1983:UF}, analytic function approximation \cite{Smirnov:1968:FCV}, and orthogonal polynomial theory \cite{Suetin:1974:POR}. The Grunsky inequalities, which are about the Faber polynomials' coefficients, were known to be related to the Fredholm eigenvalue \cite{Schiffer:1981:FEG}. Recently, the Faber polynomials were applied to compute the conformal mapping \cite{Wala:2018:CMD}.

Our results explicitly reveal the relationship among the layer potential operators, the Faber polynomials and the Grunsky coefficients. For a set of density basis functions on $\partial\Omega$, namely $\{\zeta_m\}$, which we define in terms of the curvilinear orthogonal coordinates, we derive explicit series expressions for the single layer potential and the NP operator. Similar consideration holds for the double layer potential. The results are summarized in Theorem \ref{thm:series}. One of the remarkable consequences of our approach is that the NP operator with respect to the basis $\{\zeta_m\}$ has a doubly infinite, self-adjoint matrix representation
\begin{equation}\label{eqn:matrixKstar}
\ds[\KstarOmega]=\frac{1}{2}\begin{bmatrix}
\ds& & \vdots & & & &\vdots & &\\[2mm]
\ds& 0 & 0 & 0& 0&{\mu}_{3,1}&{\mu}_{3,2}&{\mu}_{3,3}&\\[2mm]
\ds\cdots& 0 & 0 & 0& 0&{\mu}_{2,1}&{\mu}_{2,2}&{\mu}_{2,3}&\cdots\\[2mm]
\ds& 0 & 0 & 0& 0&{\mu}_{1,1}&{\mu}_{1,2}&{\mu}_{1,3}&\\[2mm]
\ds& 0 & 0 & 0& 1&0&0&0&\\[2mm]
\ds&{\overline{{\mu}_{1,3}}}& {\overline{{\mu}_{1,2}}}\ds&{\overline{{\mu}_{1,1}}}&0&0&0&0&\\[2mm]
\ds\cdots&{\overline{{\mu}_{2,3}}}&{\overline{{\mu}_{2,2}}}\ds&{\overline{{\mu}_{2,1}}}&0&0&0&0&\cdots\\[2mm]
\ds&\overline{{\mu}_{3,3}}& {\overline{{\mu}_{3,2}}}\ds&{\overline{{\mu}_{3,1}}}&0&0&0&0&\\[2mm]
\ds& &\vdots & & & &\vdots& &
\end{bmatrix}.
\end{equation}
It is worth emphasizing that $\Kcal_{\p\Om}$ and $\Kcal^*_{\p\Om}$ are identical to the same matrix operator via (two different sets of) boundary basis functions; see the discussion below Theorem  \ref{thm:series} for more details. The matrix formulation of the NP operators provides us a simple numerical scheme to compute the transmission problem solution and, also, to approximate the spectrum of $\Kcal^*_{\p\Om}$ for a smooth domain, by use of the finite section method.

The proposed method requires us to know the exterior conformal mapping coefficients. We derive an integral formula for the exterior conformal mapping coefficients.
To state the result more simply, we identify $z=x_1+ix_2$ in $\CC$ with $x=(x_1,x_2)$ in $\RR^2$. We assume that $\Psi$ maps $\{w\in\CC:|w|>\gamma\}$ conformally onto $\CC\setminus\overline{\Om}$ with $\gamma>0$ and that it has Laurent series expansion
\beq\label{conformal:Psi}
\Psi(w)=w+a_0+\frac{a_1}{w}+\frac{a_2}{w^2}+\cdots.
\eeq
From the Riemann mapping theorem there exist unique $\gamma$ and $\Psi$ satisfying such properties; see for example \cite[Chapter 1.2]{Pommerenke:1992:BBC}.
It turns out (see Theorem \ref{thm:a_k} and its proof in section \ref{sec:transmission_solution}) that the coefficients satisfy
\begin{align*}
\ds&\gamma^2 = \frac{1}{2\pi}\int_{\p\Om}z\overline{\varphi(z)}\,d\sigma(z),\\
\ds&a_m = \frac{\gamma^{m-1}}{2\pi}\int_{\p\Om}z|\varphi(z)|^{-m+1}(\varphi(z))^m\,d\sigma(z),\quad m=0,1,\dots,
\end{align*}
where $$\varphi(z)=(I-2\Kcal^*_{\p\Om})^{-1}(\nu_1+i\nu_2).$$
This formula leads to a numerical method to compute the exterior conformal mapping coefficients for a given domain of arbitrary shape (and the interior conformal mapping by reflecting the domain across a circle). It is worth mentioning that a numerical scheme for computation of the conformal mapping in terms of the double layer potential was observed in \cite{Wala:2018:CMD}.

The rest of the paper is organized as follows. In section 2 we formulate the transmission problem \eqnref{cond_eqn0} using boundary integrals. Section 3 constructs harmonic basis functions by using the Faber polynomials. We then define two separable Hilbert spaces on $\p \Om$ and obtain their properties in section 4. Section 5 is devoted to deriving series expansions of the single and double layer potentials and the NP operators. We finally investigate the properties of the NP operators in the defined Hilbert spaces and provide the numerical scheme to compute the solution to the transmission problem and the spectrum of $\Kcal^*_{\p\Om}$ in section 6, and we conclude with some discussion.

\section{Boundary integral formulation}\label{sec:integralformulation}

For $\varphi\in \LtwobdOmega$, we define 
\begin{align*}
	\SingleOmega[\varphi](x)&=\int_{\partial\Omega}\Gamma(x-y)\varphi(y)\,d\sigma(y),~~~~x\in\Rtwo,\\[1.5mm]
	\DoubleOmega[\varphi](x)&=\int_{\partial\Omega}\frac{\partial}{\partial\nu_y}\Gamma(x-y)\varphi(y)\,d\sigma(y),~~~~x\in\Rtwo\setminus\partial\Omega,
\end{align*}
where  $\Gamma$ is the fundamental solution to the Laplacian, {\it i.e.}, $$\Gamma(x)=\frac{1}{2\pi}\ln|x|$$ and $\nu_y$ denotes the outward unit normal vector on $\partial\Omega$. We call $\SingleOmega[\varphi]$ a single layer potential and $\DoubleOmega[\varphi]$ a double layer potential associated with the domain $\Om$. 
The Neumann-Poincar\'{e} (NP) operators $\Kcal_{\Om}$ and $\Kcal_{\Om}^*$ are defined as
\begin{align*}
	\ds\KstarOmega[\varphi](x)=p.v.\frac{1}{2\pi}\int_{\partial\Omega}\frac{\left<x-y,\nu_x\right>}{|x-y|^2}\varphi(y)\,d\sigma(y),\\[1.5mm]
	\ds	\KOmega[\varphi](x)=p.v.\frac{1}{2\pi}\int_{\partial\Omega}\frac{\left<y-x,\nu_y\right>}{|x-y|^2}\varphi(y)\,d\sigma(y).
\end{align*}
Here $p.v$ denotes the Cauchy principal value. One can easily see that $\KstarOmega$ is the $L^2$ adjoint of $\KOmega$.

The single and double layer potentials satisfy the following jump relations on the interface, as shown in \cite{Verchota:1984:LPR}:
\begin{align}
	\ds	\SingleOmega[\varphi]\Big|^{+}(x)&=\SingleOmega[\varphi]\Big|^{-}(x)~~~~~~~~\text{a.e. }x\in\partial\Omega,\notag\\[1.5mm]
	\ds\frac{\partial}{\partial\nu}\SingleOmega[\varphi]\Big|^{\pm}(x)&=\left(\pm\frac{1}{2}I+\KstarOmega\right)[\varphi](x)~~~~~~~~\text{a.e. }x\in\partial\Omega\label{eqn:Kstarjump},\\[1.5mm]
	\ds	\DoubleOmega[\varphi]\Big|^{\pm}(x)&=\left(\mp\frac{1}{2}I+\KOmega\right)[\varphi](x)~~~~~~~~\text{a.e. }x\in\partial\Omega\notag,\\[1.5mm]
	\ds	\frac{\partial}{\partial\nu}\DoubleOmega[\varphi]\Big|^{+}(x)&=\frac{\partial}{\partial\nu}\DoubleOmega[\varphi]\Big|^{-}(x)~~~~~~~~\text{a.e. }x\in\partial\Omega.\notag
\end{align}
Due to the jump formula \eqnref{eqn:Kstarjump}, the solution to \eqnref{cond_eqn0} can be expressed as 
\beq\label{umh}
u(x)=H(x)+\Scal_{\p \Om}[\varphi](x),\quad x\in\RR^2,
\eeq
where $\varphi$ satisfies
\beq\label{eqn:boundaryintegral_varphi}
\varphi=(\lambda I-\Kcal_{\p \Om}^*)^{-1}\left[\nu\cdot\nabla H\right]\quad\mbox{on }\p \Om
\eeq
with $\lambda =\frac{\epsilon_c+\epsilon_m}{2(\epsilon_c-\epsilon_m)}$. 
For $|\lambda|\geq 1/2$, the operator $\lambda I -\Kcal_{\p\Om}^*$ is invertible on $L^2_0(\p\Om)$ \cite{Escauriaza:1992:RTW,Kellogg:1929:FPT}; see \cite{Ammari:2013:MSM:book, Ammari:2004:RSI:book} for more details and references.  
 Note that $\lambda$ in \eqnref{eqn:boundaryintegral_varphi} belongs to the resolvent of $\Kcal^*_{\p\Om}$ for any $0<\epsilon_c/\epsilon_m\neq 1<\infty$.

Let us review some properties of the NP operators.
The operator $\Kcal_{\p\Om}$ is symmetric in $L^2(\p \Om)$ only for a disk or a ball \cite{Lim:2001:SBI}. However, $\Kcal_{\p\Om}$ and $\Kcal^*_{\p\Om}$ can be symmetrized using Plemelj's symmetrization principle (see \cite{Khavinson:2007:PVP})
\beq\label{eqn:symmetrization}
\Scal_{\p\Om}\Kcal_{\p\Om}^*=\Kcal_{\p\Om}\Scal_{\p\Om}.
\eeq
We denote by $H^{-1/2}_0(\p\Om)$ the space of functions $u$ contained in $H^{-1/2}(\p\Om)$ such that $\la u, 1\ra_{-1/2,1/2}=0$, where $\la\cdot,\cdot\ra_{-1/2,1/2}$ is the duality pairing between the Sobolev spaces $H^{-1/2}(\p\Om)$ and $H^{1/2}(\p\Om)$. The operator $\Kcal_{\p\Om}^*$ is self-adjoint in $\mathcal{H}^*$ which is the space  $H^{-1/2}_0(\p\Om)$ equipped with the new inner product 
\beq
\la \varphi,\psi\ra_{\mathcal{H}^*}:=-\la \varphi,\Scal_{\p\Om}[\psi]\ra_{-1/2,1/2}.
\eeq
 The spectrum of $\Kcal_{\p\Om}^*$ on $\mathcal{H}^*$ lies in $(- 1/2 , 1/2)$ \cite{Escauriaza:2004:TPS, Fabes:1992:SRC, Kellogg:1929:FPT}; see also \cite{Kang:2018:SPS,Krein:1998:CLO} for the permanence of the spectrum for the NP operator with different norms. If $\p\Om$ is $C^{1,\alpha}$ with some $\alpha>0$, then $\mathcal{K}^*_{\p D}$ is compact as well as self-adjoint on $\mathcal{H}^*$. Hence, $\Kcal_{\p\Om}^*$ has discrete real eigenvalues contained in $(-1/2,1/2)$ that accumulate to $0$.

Plasmonic materials, of which permittivity has a negative real part and a small loss parameter, admit the so-called plasmonic resonance when the corresponding $\lambda$ is very close to the spectrum of $\Kcal_{\p \Om}^*$.
We refer the reader to \cite{Ando:2016:APR, Bonnetier:2018:PRB:preprint, Helsing:2017:CSN,Helsing:2013:PCC,Kang:2017:SRN, Perfekt:2017:ESN} and references therein for recent results on the spectral properties of the NP operators and plasmonic resonance.
%

\section{Geometric harmonic basis}
In this section, we construct a set of harmonic basis functions based on the exterior conformal mapping and the Faber polynomials. We introduce only the key properties of Faber polynomials in this section; however, we offer more information in the appendix in order to make this paper self-contained.
\subsection{Faber Polynomials}\label{sec:Faber}
	Let $z=\Psi(w)$ be the exterior conformal mapping associated with $\Om$ given by \eqnref{conformal:Psi}. The mapping $\Psi$ uniquely defines a sequence of $m$-th order monic polynomials $\{F_m(z)\}_{m=0}^\infty$, called the Faber polynomials, via the generating function relation
\beq\label{eqn:Fabergenerating}
\frac{\Psi'(w)}{\Psi(w)-z}=\sum_{m=0}^\infty \frac{F_m(z)}{w^{m+1}},\quad z\in\overline{\Om},\ |w|>\gamma.
\eeq
In what follows, $\partial\Omega_r$ denotes the image of $|w|=r~(r\geq\gamma)$ under the mapping $\Psi(w)$ and $\Omega_r$ is the region enclosed by $\partial\Omega_r.$
For every fixed $z\in\overline{\Omega_r}$ with $r\geq\gamma$,
the series \eqref{eqn:Fabergenerating} converges in the domain $|w|>r$ and, furthermore, it uniformly converges 
in the closed domain $|w|\geq r$ if $z\in\Omega_r$ (see \cite{Smirnov:1968:FCV} for more details).
Let us state the properties of Faber polynomials that are the key technical tools of this paper (see appendix \ref{section:appdixFaber} for the derivation).

\begin{itemize}
\item Decomposition of the fundamental solution to the Laplacian:
 \begin{equation}\label{eqn:log_decomp}
\log(\Psi({w})-{z})=\log w -\sum_{m=1}^\infty \frac{1}{m}F_m({z})w^{-m},\quad  |w|>r,\ z\in\Omega_r.
\end{equation}
One can derive this equation by integrating \eqnref{eqn:Fabergenerating} with respect to $w.$

\item Series expansion in the region $\mathbb{C}\setminus\overline{\Omega}$:
 \begin{equation}\label{eqn:Faberdefinition}
	F_m(\Psi(w))
	=w^m+\sum_{k=1}^{\infty}c_{m,k}{w^{-k}},\quad m=1,2,\dots.
\end{equation}
The coefficients $c_{m,k}$ are called the Grunsky coefficients.

\item Grunsky identity:
\begin{equation}\label{eqn:Grunskyidentity}
	k c_{m,k}=m c_{k,m}\quad \text{for any } m,k\geq1.	
\end{equation}

\item Bounds on the Grunsky coefficients:
\begin{equation}\label{eqn:GrunskyBounds}
\sum_{k=1}^\infty \left|  \sqrt{\frac{k}{m}}\frac{c_{m,k}}{\gamma^{m+k}} \right|^2\leq 1\quad \text{for any } m\geq 1.
\end{equation}

\end{itemize}

\begin{remark}
	Once the coefficients of the exterior conformal mapping $\gamma, a_0, a_1, a_2\dots$ are known, the	Faber polynomials and the Grunsky coefficients can be easily  computed via the recursion formulas
	\eqnref{eqn:Faberrecursion} and \eqnref{eqn:cnkrecursion} in the appendix.
\end{remark}

\subsection{Geometric harmonic basis functions}\label{sec:geometric_basis}

While it is straightforward to find appropriate harmonic basis functions for a circle or an ellipse, it becomes complicated for a domain with arbitrary geometry. Since the solution to \eqnref{cond_eqn0} is harmonic in each of the two regions $\Om$ and $\CC\setminus\overline{\Om}$, we require two sets of basis functions for the interior and exterior of $\Om$, separately. We construct two systems satisfying the following:
\begin{itemize}
\item[(a)] {\it Interior harmonic basis}, consisting of polynomial functions, such that any harmonic function in a domain containing $\overline{\Om}$ can be expressed as a series of these basis functions. 
\item[(b)]  {\it Exterior harmonic basis}, consisting of harmonic functions in $\mathbb{C}\setminus\overline{\Om}$, such that any harmonic function in $\CC\setminus\overline{\Om}$ which decays like $O(|x|^{-1})$ as $|x|\rightarrow\infty$ can be expressed as a series of these basis functions.
\end{itemize}
In addition, we require the two sets of basis functions to have explicit relations on $\partial \Omega$ such that the interior and exterior boundary values of the solution can be matched.

We remind the reader that the Faber polynomials are monic and admit the series expansion in the exterior region $\mathbb{C}\setminus\overline{\Omega}$ in terms of $w^{\pm k}$, $k\in\NN$, as shown in \eqnref{eqn:Faberdefinition}. Each $w^{-k}$ decays to zero as $|z|\rightarrow\infty$ since 
$
w=\Psi^{-1}(z)=z+O(1)
$
 and is harmonic (see \eqref{eqn:laplacian}). Because of these reasons we set
  \begin{itemize}
\item[(a)] {\it Interior harmonic basis}: $\left\{F_m(z)\right\}_{m\in\NN},$
\item[(b)]  {\it Exterior harmonic basis}: $\left\{w^{m}(z)\right\}_{m\in\ZZ}$.
\end{itemize}
Here, $w^m(z)$ means the function $\left(\Psi^{-1}(z)\right)^m$.

\section{Two Hilbert spaces on $\partial\Omega$: definition and duality}\label{sec:Kpm:def}
In the previous section we defined the harmonic basis functions in the interior and exterior of the domain $\Om$. We recall the reader that the boundary integral formulation \eqnref{umh} is involved with the density function $\varphi$ on $\partial\Omega$ satisfying \eqnref{eqn:boundaryintegral_varphi}. In this section we construct a basis for density functions on $\p\Om$ with which one can reformulate \eqnref{umh} and \eqnref{eqn:boundaryintegral_varphi} in series form.

First, we introduce the curvilinear coordinate system in the exterior region $\mathbb{C}\setminus{\overline{\Omega}}$ associated with the exterior conformal mapping $\Psi$. Then, we construct basis functions on $\p\Om$ with the coordinates and define two Hilbert spaces $K^{\pm 1/2}(\partial\Omega)$ that are dual to each other and can be identified with $l^2(\CC)$ via the Fourier series expansions.

\subsection{Orthogonal coordinates in $\mathbb{C}\setminus \Omega$}\label{sec:coordinate}

In view of the domain representation using its exterior conformal mapping $\Psi$, it is natural to adopt the curvilinear coordinate system generated by $z=\Psi(w)$. To deal with the transmission condition on $\partial\Omega$ in terms of $\Psi$, the regularity of $\Psi$ up to the boundary $\p\Om$ should be assumed. The continuous extension of the conformal mapping to the boundary is well known (Carath\'{e}odory theorem \cite{Caratheodory:1913:GBR}). When the boundary $\partial\Omega$ is $C^{1,\alpha}$, the exterior conformal mapping $\Psi$ allows the $C^{1,\alpha}$ extension to $\partial\Omega$ by Kellogg-Warschawski theorem \cite[Theorem 3.6]{Pommerenke:1992:BBC}.
If $\Om$ has a corner point, then the regularity of $\Psi$ depends on the angle of $\p\Om$ at the corner point.
The regularity of the interior conformal mapping for a domain with corners is well established (see for example \cite{Pommerenke:1992:BBC}), and the results can be equivalently translated into the exterior case.
We provide the regularity result for the exterior conformal mapping $\Psi$ in terms of the exterior angles of corner points in appendix \ref{appen:boundarybehavior}.
Figure \ref{fig:angle} illustrates the exterior angle at a corner point.
\begin{figure}[!h]
\centering
\includegraphics[width=5cm]{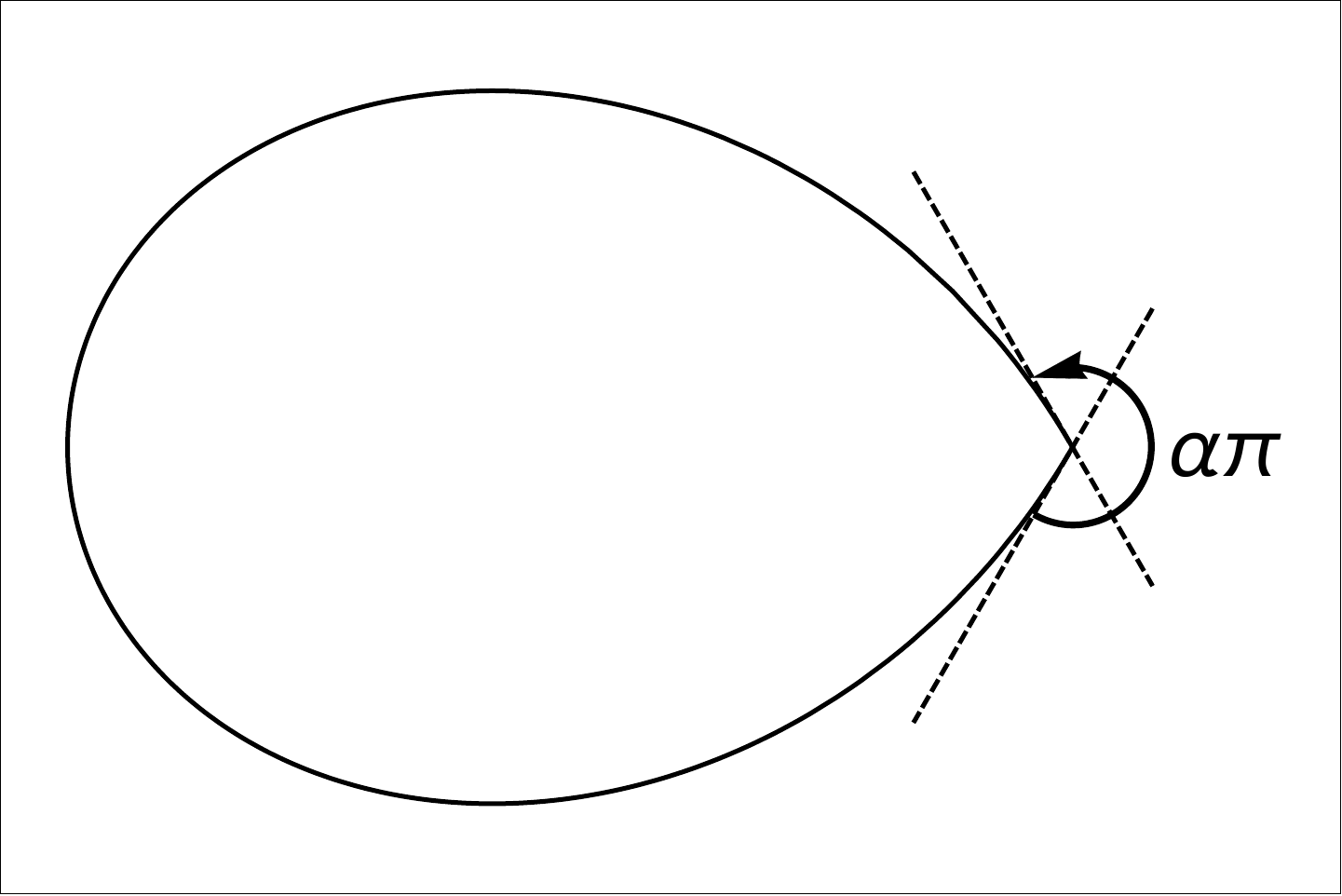}
\caption{A domain with one corner point whose exterior angle is $\alpha \pi$.}\label{fig:angle}
\end{figure}

Here and after, we assume that the boundary $\partial\Omega$ is a piecewise $C^{1,\alpha}$ Jordan curve, possibly with a finite number of corner points without inward or outward cusps. We define the coordinate system which associates each $z\in \CC\setminus\Om$ with the modified polar coordinate $(\rho,\theta)\in[\rho_0,\infty)\times[0,2\pi)$ via the relation
$$z=\Psi(e^{\rho+i\theta}).$$
We let $\Psi(\rho,\theta)$ to indicate $\Psi(e^{\rho+i\theta})$ for convenience.

Denote the scale factors as $h_\rho=|\frac{\partial \Psi}{\partial\rho}|$ and $h_\theta=|\frac{\partial \Psi}{\partial\theta}|$. 
The partial derivatives satisfy $i\frac{\partial \Psi}{\partial \rho}=\frac{\partial \Psi}{\partial \theta}$ so that $\{\frac{\partial \Psi}{\partial\rho},\frac{\partial \Psi}{\partial\theta}\}$  are orthogonal vectors in $\CC$ and the scale factors coincide.
 We set $$h(\rho,\theta):=h_\rho=h_\theta.$$ We remark that the scale factor $h(\rho,\theta)$ is integrable on the boundary $\partial\Omega$ (for the proof see Lemma \ref{lemma:regularity} in the appendix).
One can easily show, for a function $u$ defined in the exterior of $\Om$, that 
\beq
	\Delta u=\frac{1}{h^2(\rho,\theta)}\left(\frac{\partial^2 u}{\partial \rho^2}+\frac{\partial^2 u}{\partial \theta^2}\right).\label{eqn:laplacian}
\eeq

On $\p\Om=\{\Psi(\rho_0,\theta):\theta\in[0,2\pi)\}$, the length element is $d\sigma(z)=h(\rho_0,\theta)d\theta$ for $z=\Psi(\rho_0,\theta)$. The exterior normal derivative of $u(z)=(u\circ\Psi)(\rho,\theta)$ is
 \beq
 	\frac{\partial u}{\partial \nu}\Big|_{\p\Om}^{+}(z)=\frac{1}{h}\frac{\partial }{\partial \rho}u(\Psi(e^{\rho+i\theta}))\Big|_{\rho\rightarrow\rho_0^+}.\label{eqn:normalderiv}
 \eeq
 A great advantage of using the coordinate system $(\rho,\theta)$ is that, thanks to $h_\rho=h_\theta$, the integration of the normal derivative for $u$ is simply
	\begin{equation}\label{eqn:boundaryintegral}
		\int_{\partial\Omega}\frac{\partial u}{\partial\nu}\Big|^+_{\p\Om}(z)\, d\sigma(z)=\int_{0}^{2\pi}\frac{\partial u}{\partial \rho}(\rho_0,\theta)\Big|_{\rho\rightarrow\rho_0^+}\, d\theta.
	\end{equation}

\subsection{Geometric density basis functions}
In this subsection, we set up two systems of density basis functions on $\p\Om$ whose usage will be clear in the subsequent sections.
 
 Define for each $m\in\ZZ$ the density functions
 \beq
 \begin{cases}
 \ds\teta_m(z)=\ds\teta_m(\Psi(e^{\rho_0+i\theta}))=e^{im\theta},\\
 \ds\tzeta_m(z)=\ds\tzeta_m(\Psi(e^{\rho_0+i\theta}))=\frac{e^{im\theta}}{h(\rho_0,\theta)}.
 \end{cases}
 \eeq
 We then normalize them (with respect to the norms that will be defined later) as
 \beq
 \begin{cases}
 \ds \eta_m(z)=\ds|m|^{-\frac{1}{2}}\teta_m(z),\\[2mm]
 \ds\zeta_m (z)= \ds|m|^{\frac{1}{2}}\;\tzeta_m(z),\quad m\neq 0.
 \end{cases}
 \eeq
 For $m=0$, we set $\zeta_0=\tzeta_0$ and $\eta_0=\teta_0=1$. 
  Due to Lemma \ref{lemma:regularity} in the appendix, we have $h(\rho_0,\theta),\frac{1}{h(\rho_0,\theta)}\in L^1([0,2\pi])$, and, hence,
 \beq\label{zetaL2}\tzeta_m(z),\ \teta_m(z),\ \zeta_m(z),\ \eta_m(z)\in L^2(\p\Om).\eeq
 
 In Figure \ref{fig:kite}, two geometric boundary basis functions $\teta_1$ and $\tzeta_1$ are drawn for a domain enclosed by a parametrized curve, where the corresponding conformal mapping is computed by using   
 Theorem \ref{thm:a_k}.

 Before defining new spaces on $\p\Om$, let us consider the Sobolev spaces $H^{\pm 1/2}$ on the $1$-dimensional torus $\TT^1=\RR^1/2\pi\ZZ^1$.
 We denote by $L^2(\TT^1)$ the space consisting of periodic functions $f$ on $\TT^1$ such that $$\|f\|^2_{L^2(\TT^1)}=\frac{1}{2\pi}\int_0^{2\pi}|f(\theta)|^2 d\theta<\infty.$$
 The space $L^2(\TT^1)$ can be identified with $l^2(\ZZ)$ via the Fourier basis. Similarly, the Sobolev space $H^{1/2}(\TT^1)$ admits the Fourier series characterization as follows:
 $$H^{1/2}(\TT^1)=\left\{\varphi =\sum_{m=-\infty}^\infty a_m e^{im\theta}\;\bigg|\;\|\varphi\|_{H^{1/2}(\TT^1)}^2=|a_0|^2+\sum_{m=-\infty,m\neq0}^\infty |m||a_m|^2<\infty\right\}.$$
 For each $l\in H^{-1/2}(\TT^1)=\left(H^{1/2}(\TT^1)\right)^*$, it satisfies
 $$\|l\|_{H^{-1/2}(\TT^1)}^2=|b_0|^2+\sum_{m=-\infty, m\neq0}^\infty |m|^{-1}|b_m|^2<\infty,\quad\mbox{where } b_m = l(e^{im\theta}).$$
 Conversely, for each sequence $(b_m)\in l^2(\ZZ)$ satisfying $|b_0|^2+\sum_{m=-\infty, m\neq 0}^\infty |m|^{-1}|b_m|^2<\infty$, there exists $l\in H^{-1/2}(\TT^1)$ such that $l(e^{im\theta})=b_m$ for each $m$. 
 We will define two separable Hilbert spaces on $\p\Om$ in a similar manner in the following subsection. 
 \vskip .5cm

 \begin{figure}[!b]
     \begin{center}
     \begin{subfigure}{.475\textwidth}
 	  	\centering
 	  	\includegraphics[height=4.5cm,width=6cm, trim={0 1cm 0 0.8cm}, clip]{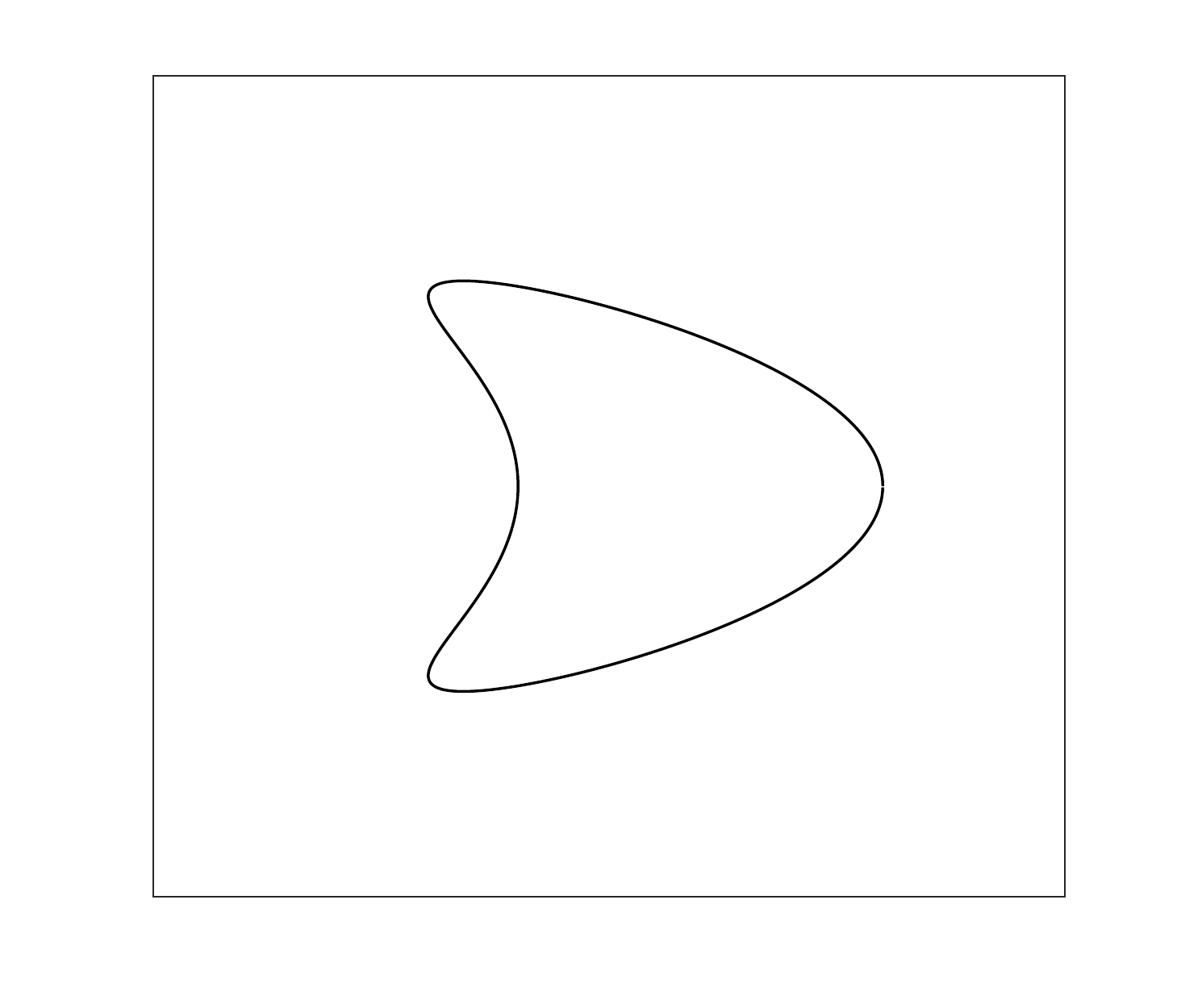}
 	  	\caption{Kite-shaped domain $\Om$}
 	\end{subfigure}
 	\hfill
 	\begin{subfigure}{.475\textwidth}
 	  	\centering
 	  	\includegraphics[height=4.5cm,width=6cm,  trim={0.5cm 1cm 0 0.8cm}, clip]{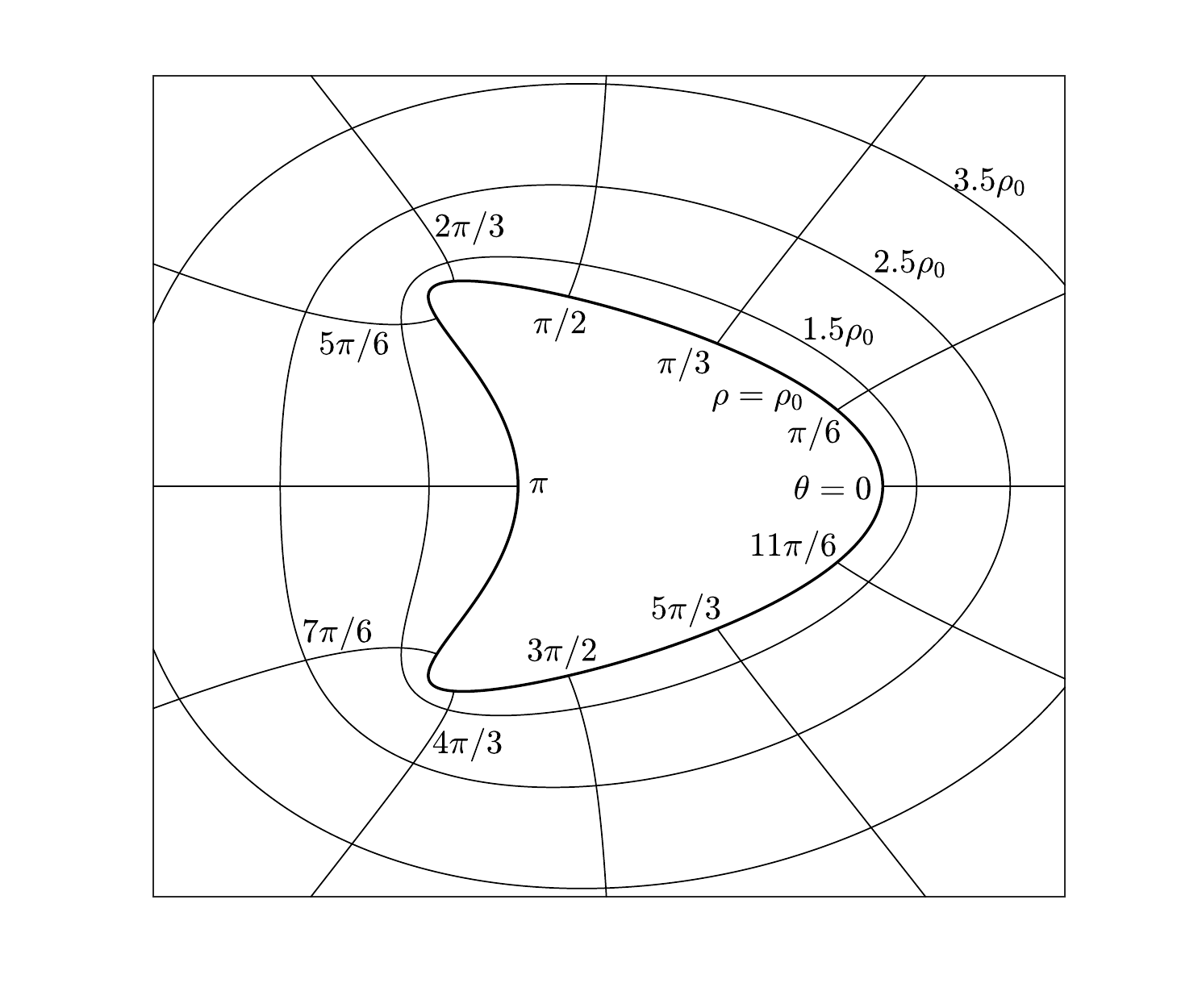}
 	  	\caption{Level coordinates curves of $\Psi(\rho,\theta)\}$}
 	\end{subfigure}
     \begin{subfigure}{.475\textwidth}
 		\centering
 	    \includegraphics[height=4.5cm,width=6cm,  trim={0.3cm 0.1cm 0.3cm 0.5cm}, clip]{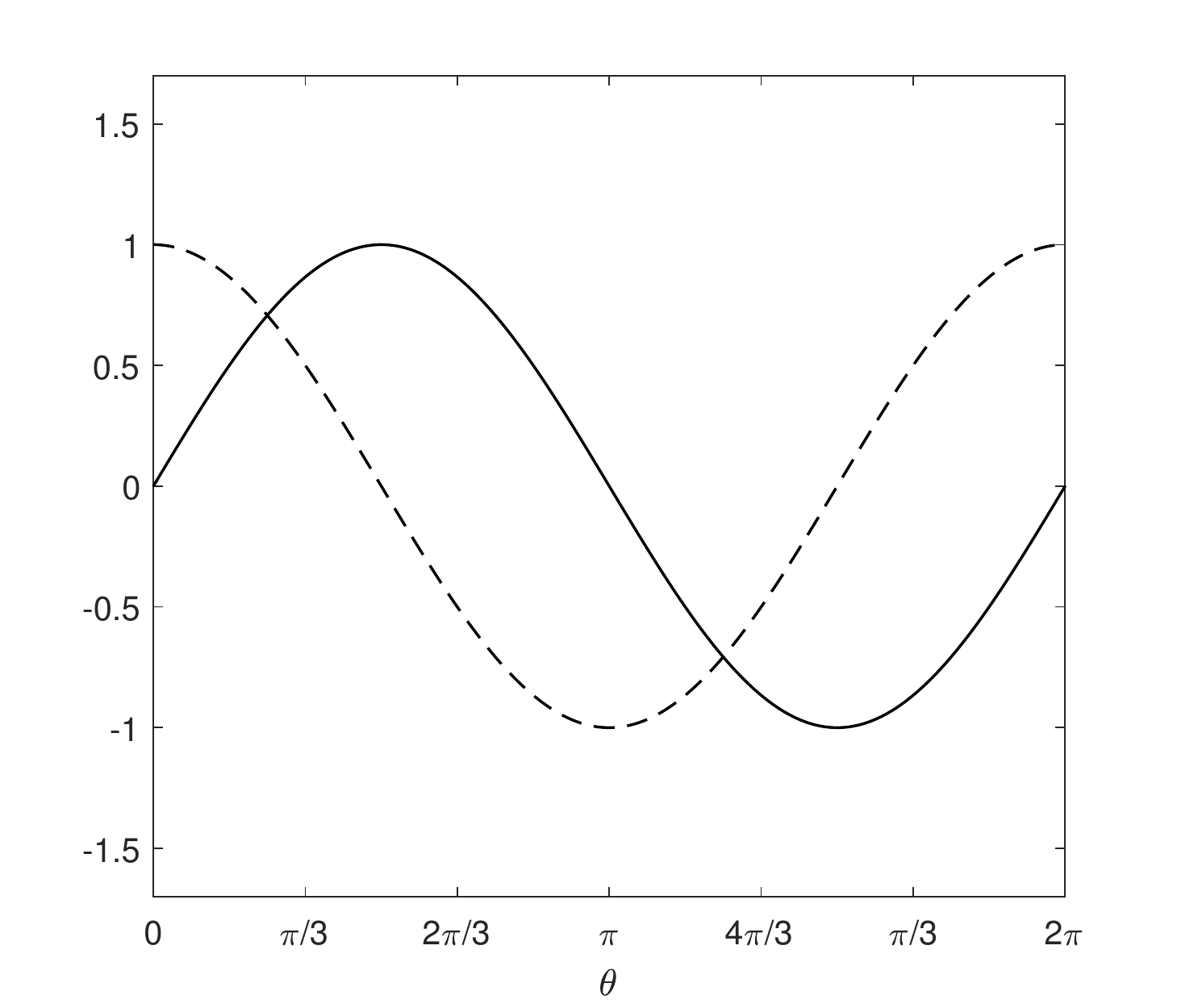}
 	    \caption{$\teta_1(\theta)=e^{i\theta}$}
 	\end{subfigure}%
 \hskip 0.5cm
 	\begin{subfigure}{.475\textwidth}
 	  	\centering
 	  	\includegraphics[height=4.5cm,width=6cm,  trim={0.3cm 0.1cm 0.3cm 0.5cm}, clip]{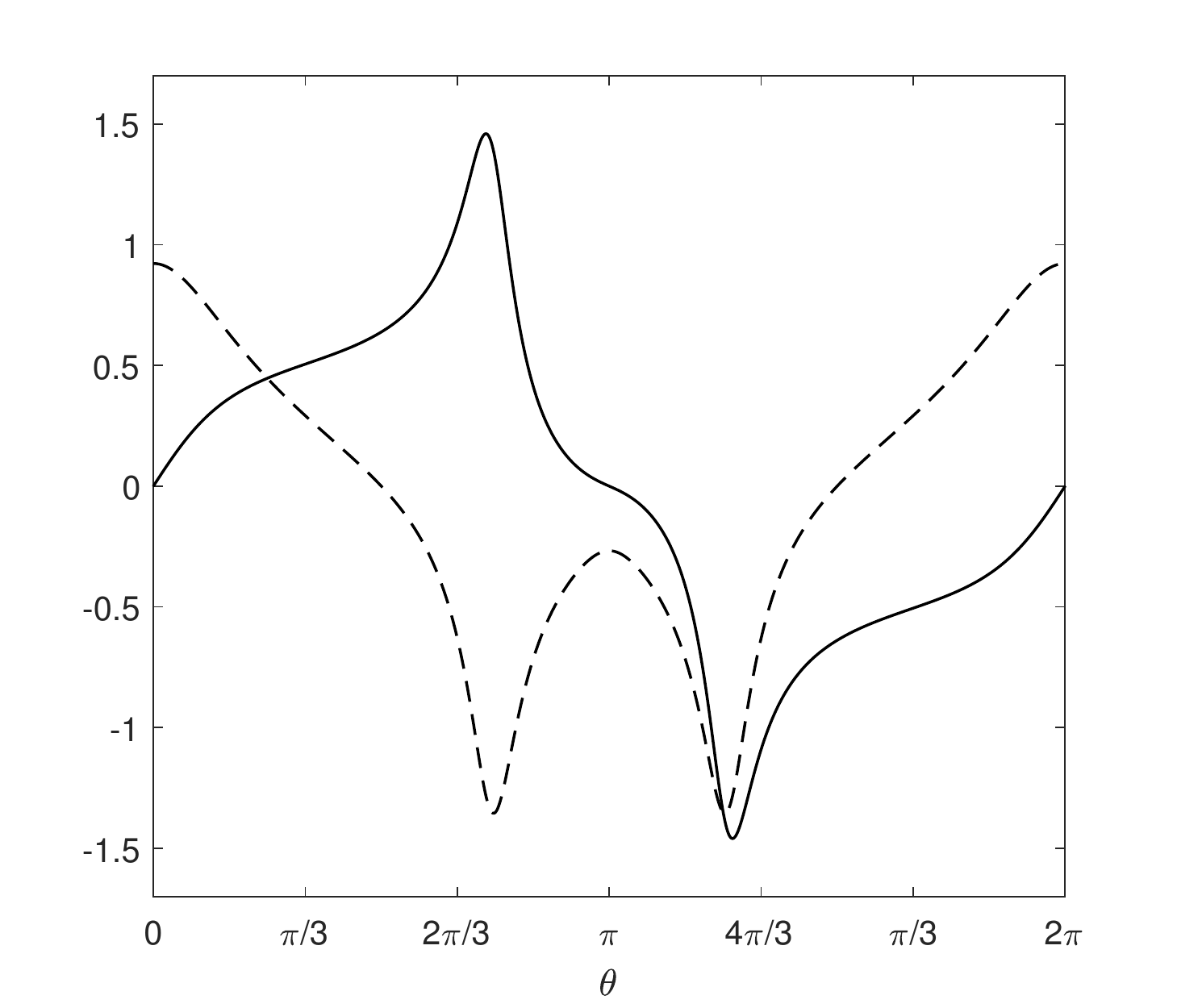}
 		 \caption{$\tzeta_1(\theta)=e^{i\theta}/h(\theta)$}
 	\end{subfigure}
 	\vskip -.5cm
    \caption{A general shaped domain $\Om$ and its geometric boundary basis functions. The domain $\Om$ is given by the parametrization $x(t) = \cos(t) + 0.65\cos(2t) - 0.65,\ y(t) = 1.5\sin(t)$, $t\in[0,2\pi]$. (a) illustrates $\p\Om$.
    (b) shows several level curves of curvilinear coordinates $(\rho,\theta)$ made by the exterior conformal mapping associated with $\Om$, where the coefficients of $\Psi$ are numerically computed by using Theorem \ref{thm:a_k}. 
      (c,d) show the real (dashed) and imaginary part (solid) of the basis functions.
 }\label{fig:kite}
 \end{center}
 \end{figure}
%

\subsection{Definition of the spaces $\Keta(\p\Om)$ and $\Kzeta(\p\Om)$}
For the sake of simplicity, we write $f(\theta)=(f\circ\Psi)(\rho_0,\theta)$ for a function $f$ defined on $\partial\Omega$.

Consider the vector space of functions
\begin{align}\label{def:Kzeta}
K^{-1/2}(\p\Om)&:=\left\{\varphi:\partial\Omega\rightarrow\mathbb{C}\;\Big| \;\ \sum_{m\in\ZZ}|a_m|^2<\infty,~a_m=\frac{1}{2\pi}\int_{\partial\Omega}\varphi\overline{\eta_{m}} \,d\sigma\right\}.
\end{align}
We shall consider two functions $\varphi_1,\varphi_2\in K^{-1/2}(\partial\Omega)$ equivalent if 
$$\frac{1}{2\pi}\int_{\partial\Omega}\varphi_1\overline{\eta_m }\, d\sigma=\frac{1}{2\pi}\int_{\partial\Omega}\varphi_2\overline{\eta_m }\, d\sigma\quad \text{for all } m\in\mathbb{Z}.$$
We do not distinguish between equivalent functions in $K^{-1/2}(\partial\Omega)$. Among all functions in the equivalence class, denoted by $[\varphi]$, containing a given element $\varphi\in K^{-1/2}(\partial\Omega)$, we take the series expansion with respect to the basis $\{\zeta_m\}$ as the representative of the class $[\varphi]$. In other words, we write
$$\varphi=\sum_{m\in\mathbb{Z}}a_m\zeta_m \quad\mbox{with} \quad a_m=\frac{1}{2\pi}\int_{\partial\Omega}\varphi\overline{\eta_{m}}\, d\sigma.$$
Then one can define the inner product and the associated norm in $K^{-1/2}(\p\Om)$ in terms of the Fourier coefficients with respect to the basis $\{\zeta_m\}.$ In the same way we define $K^{1/2}(\p\Om)$, by exchanging the role of $\{\zeta_m\}$ and $\{\eta_m\}$, as
\begin{align}\label{def:Keta}
K^{1/2}(\p\Om)&:=\left\{\psi:\partial\Omega\rightarrow\mathbb{C}\;\Big| \;\ \sum_{m\in\ZZ}|b_m|^2<\infty,~b_m=\frac{1}{2\pi}\int_{\partial\Omega}\varphi\overline{\zeta_{m}} \,d\sigma\right\}.
\end{align}
For any $\psi\in K^{1/2}(\p\Om)$, we can write
$$\psi= \sum_{m\in\ZZ}b_m \eta_m\quad \mbox{with}\quad b_m = \frac{1}{2\pi}\int_{\partial\Omega}\varphi\overline{\zeta_{m}}\, d\sigma.$$
We identify the two spaces $K^{-1/2}(\p\Om)$ and $K^{1/2}(\p\Om)$ with $l^2(\CC)$ and define the inner-products via the boundary bases $\{\zeta_m\}$ and $\{\eta_m\}$, respectively. The discussion can be summarized as follows. 
\begin{definition}\label{definition:twospaces}
We define two Hilbert spaces $K^{-1/2}(\p\Om)$ and $K^{1/2}(\p\Om)$ by \eqnref{def:Kzeta} and \eqnref{def:Keta} (quotiented by the equivalence class of the zero function) such that they are isomorphic to $l^2(\CC)$ via the boundary bases $\{\zeta_m\}$ and $\{\eta_m\}$, respectively.
In other words, they are
\begin{align*}
K^{-1/2}(\p\Om)&=\left\{\varphi=\sum_{m\in\ZZ} a_m \zeta_m\;\Big|\;\ \sum_{m\in\ZZ} |a_m|^2<\infty\right\},\\
K^{1/2}(\p\Om)&=\left\{\psi=\sum_{m\in\ZZ}b_m \eta_m\;\Big| \;\  \sum_{m\in\ZZ} |b_m|^2<\infty\right\}
\end{align*}
equipped with the inner products
\begin{align}
&\Big(\sum c_m\zeta_m,\ \sum d_m \zeta_m\Big)_{-1/2}=\sum c_m \overline{d_m},\\\label{innerproduct_zeta}
&\Big(\sum c_m\eta_m,\ \sum d_m \eta_m\Big)_{1/2}=\sum c_m \overline{d_m}.
\end{align}
For the sake of notational convenience we may simply write $\Kzeta$ and $\Keta$ for the two spaces.
\end{definition}

Let us consider the operator $$I(\varphi,\psi)=\frac{1}{2\pi}\int_{\partial\Omega} \varphi(z)\overline{\psi(z)}\,d\sigma(z)\quad \mbox{for }\varphi\in K^{-1/2},\  \psi\in\Keta.$$  
For any finite combinations of basis functions it holds that
$I\left(\sum_{|m|\leq N}a_m\zeta_m, \sum_{|m|\leq N}b_m\eta_m\right)=\sum_{|m|\leq N} a_m \overline{b_m},$
and hence we have $|I(\varphi,\psi)|\leq \|\varphi\|_{\Kzeta}\|\psi\|_{\Keta}<\infty$. 
We define a duality pairing between $K^{-1/2}(\partial\Omega)$ and $K^{1/2}(\partial \Omega)$, which is clearly the extension of the $L^2$ pairing:
$$(\varphi,\psi)_{-1/2,1/2}=\sum_{m=-\infty}^\infty a_m \overline{b_m}$$
for $\varphi=\sum a_m\zeta_m\in \Kzeta(\p\Om)$ and $\psi=\sum b_m\eta_m\in \Keta(\p\Om)$. Clearly, the pair of indexed families of functions $\{\zeta_m\}$ and $\{\eta_m\}$ is a complete biorthogonal system for $K^{1/2}(\partial\Omega)$ and $K^{-1/2}(\partial\Omega)$.

If the boundary $\partial\Omega$ is smooth enough, the space $K^{\pm1/2}(\partial\Omega)$ coincides with the classical trace spaces $H^{\pm1/2}(\partial\Omega)$. \begin{lemma}\label{smooth:equi}
Let $\Om$ be a simply connected bounded domain with $C^{1,\alpha}$ boundary with some $\alpha>0$. Then the following relations hold:
\begin{align*}
K^{1/2}(\p\Om)&=H^{1/2}(\p\Om),\\
K^{-1/2}(\p\Om)&=H^{-1/2}(\p\Om).
\end{align*}
The norm $\|\cdot\|_{\Keta(\p\Om)}$ is equivalent to $\|\cdot\|_{H^{1/2}(\p\Om)}$ and the norm $\|\cdot\|_{\Kzeta(\p\Om)}$
to $\|\cdot\|_{H^{-1/2}(\p\Om)}$.
Moreover, the two duality pairings $(\cdot,\cdot)_{-1/2,1/2}$ and $\la\cdot,\cdot\ra_{-1/2,1/2}$ coincide.
\end{lemma}
\begin{proof}
For a general domain $D$, the space $H^{1/2}(\partial D)$ can be characterized as the Hilbert space of functions $u:\partial D\rightarrow\mathbb{C}$ equipped with the fractional Sobolev-Slobodeckij norm
$$\|u\|_{H^{1/2}(\partial D)}^2= \|u\|_{L^2(\partial D)}^2+\int_{\partial D}\int_{\partial D}\frac{|u(z)-u(\tilde{z})|^2}{|z-\tilde{z}|^2 }\, d\sigma(z)d\sigma(\tilde{z})<\infty.$$
Since $h$ and $1/h$ are non-vanishing continuous on $\partial\Omega$, we deduce that $u\in H^{1/2}(\partial\Omega)$ if and only if $(u\circ \Psi)(\rho_0,\cdot)\in H^{\frac{1}{2}}(\mathbb{T}^1)$ and 
 $\|u\|_{H^{1/2}(\partial\Omega)}\sim \|(u\circ\Psi)(\rho_0,\cdot)\|_{H^{1/2}(\mathbb{T}^1)}.$
Therefore, we prove the lemma by considering the Fourier coefficients characterizations of $H^{1/2}(\mathbb{T}^1)$. 
\end{proof}

\section{Boundary integral operators in terms of geometric basis}
In this section we derive the series expansions in terms of harmonic basis functions for the boundary integral operators related to the integral formulation for the transmission problem. We then apply the results to obtain an explicit formula for the exterior conformal mapping coefficients.

\subsection{Main results}
 We set $\SingleOmega[\varphi](z)=\SingleOmega[\varphi](x)$ for $x=(x_1,x_2)$ and $z=x_1+ix_2$, and other integral operators are defined in the same way.
Here we present our main results. The proof is at the end of this subsection.

\begin{theorem}[Series expansion for the boundary integral operators]\label{thm:series}
Assume that $\Om$ is a simply connected bounded domain in $\RR^2$ enclosed by a piecewise $C^{1,\alpha}$ Jordan curve, possibly with a finite number of corner points without inward or outward cusps.
Let $F_m $ be the $m$-th Faber polynomial of $\Omega$, $c_{i,j}$ be the Grunsky coefficients and $z=\Psi(w)=\Psi(e^{\rho+i\theta})$ for $\rho>\rho_0=\ln \gamma$. 
\begin{itemize}
\item[\rm(a)]
We have (for $m=0$)
\beq\label{Scal_zeta0}
\Scal_{\p\Om}[\tzeta_0](z)=
\begin{cases}
\ln \gamma \quad &\mbox{if }z\in\overline{\Om},\\
\ln|w|\quad&\mbox{if }z\in\CC\setminus\overline{\Om}.
\end{cases}
\eeq 
For $m=1,2,\dots$, we have
	\begin{align}\label{eqn:seriesSLpositive}
		\SingleOmega[\tzeta_m](z)&=
		\begin{cases}
			\ds-\frac{1}{2m\gamma^m}F_m(z)\quad&\text{for }z\in\overline{\Omega},\\[2mm]
			\ds-\frac{1}{2m\gamma^m}\left(\sum_{k=1}^{\infty}c_{m,k}e^{-k(\rho+i\theta)}+\gamma^{2m}e^{m(-\rho+i\theta)}\right)\quad &\text{for } z\in\CC\setminus\overline{\Om},
		\end{cases}\\[3mm]
	\label{eqn:seriesSLnegative}
		\SingleOmega[\tzeta_{-m}](z)&=
		\begin{cases}
		\ds	-\frac{1}{2m\gamma^{m}}\overline{F_{m}(z)}\quad&\text{for }z\in\overline{\Omega},\\[2mm]
		\ds-\frac{1}{2m\gamma^m}\left(\sum_{k=1}^{\infty}\overline{c_{m,k}}e^{-k(\rho-i\theta)}+\gamma^{2m}e^{m(-\rho-i\theta)}\right)\quad &\text{for } z\in\CC\setminus\overline{\Om}.		
		\end{cases}
	\end{align}
	The series converges uniformly for all $(\rho,\theta)$ such that $\rho\geq\rho_1>\rho_0$.

\item[\rm(b)]
We have (for $m=0$)
\begin{align}\label{eqn:seriesDLzero}
	\ds	\DoubleOmega [1](z)&=
		\begin{cases}
		\ds	1&\text{for }z\in\Omega,\\
		\ds	0&\text{for } z\in\CC\setminus\overline{\Om}.
		\end{cases}
	\end{align}
For $m=1,2,\dots$, we have
\begin{align}\label{eqn:seriesDLpositive}
			\ds	\DoubleOmega [\teta_m](z)
	&=
		\begin{cases}
		\ds	\frac{1}{2\gamma^m}F_m(z)\quad&\text{for }z\in{\Omega},\\[2mm]
		\ds	\frac{1}{2\gamma^m}\left(\sum_{k=1}^{\infty}c_{m,k}e^{-k(\rho+i\theta)}-\gamma^{2m}e^{m(-\rho+i\theta)}\right) \quad&\text{for } z\in\CC\setminus\overline{\Om},
		\end{cases}\\[3mm]
\label{eqn:seriesDLnegative}
	\ds	\DoubleOmega [\teta_{-m}](z)&=
		\begin{cases}
		\ds	\frac{1}{2\gamma^m}\overline{F_m(z)}\quad&\text{for }z\in{\Omega},\\[2mm]
		\ds	\frac{1}{2\gamma^m}\left(\sum_{k=1}^{\infty}\overline{c_{m,k}}e^{-k(\rho-i\theta)}-\gamma^{2m}e^{m(-\rho-i\theta)}\right) \quad&\text{for } z\in\CC\setminus\overline{\Om}.
		\end{cases}
	\end{align}
	The series converges uniformly for all $(\rho,\theta)$ such that $\rho\geq\rho_1>\rho_0$.

\item[\rm(c)]
 We have (for $m=0$)
	\beq\label{eqn:Kcal:zeta0}
	\Kcal^*_{\p\Om}[\zeta_0]=\frac{1}{2}\zeta_0,
	\quad
		\KOmega [1]=\frac{1}{2}.
	\eeq
For $m =1,2,\cdots$
	\begin{align}\label{NP_series1}
		&\KstarOmega[{\zeta_m}](\theta)=\frac{1}{2}\sum_{k=1}^{\infty}\frac{\sqrt{m}}{\sqrt{k}}\frac{c_{k,m}}{\gamma^{m+k}}\, {\zeta}_{-k}(\theta),\quad
		\KstarOmega[{\zeta}_{-m}](\theta)=\frac{1}{2}\sum_{k=1}^{\infty}\frac{\sqrt{m}}{\sqrt{k}}\frac{\overline{c_{k,m}}}{\gamma^{m+k}}\, \zeta_{k}(\theta),\\
\label{NP_series2}
		&\KOmega [\eta_m](\theta)=\frac{1}{2}\sum_{k=1}^{\infty}\frac{\sqrt{k}}{\sqrt{m}}\frac{c_{m,k}}{\gamma^{m+k}}\eta_{-k}(\theta),
		\quad
		\KOmega [\eta_{-m}](\theta)=\frac{1}{2}\sum_{k=1}^{\infty}\frac{\sqrt{k}}{\sqrt{m}}\frac{\overline{c_{m,k}}}{\gamma^{m+k}}\eta_{k}(\theta).
	\end{align}
 The infinite series converges either in $\Keta(\p\Om)$ or in $\Kzeta(\p\Om)$. 
	\end{itemize}
\end{theorem}

The coefficients in the equations \eqnref{NP_series1} and \eqnref{NP_series2} are symmetric due to the Grunsky identity \eqnref{eqn:Grunskyidentity}. In other words, the double indexed coefficient 
\begin{equation}\label{def:mu}
	{\mu}_{k,m}=\sqrt{\frac{m}{k}} \frac{c_{k,m}}{\gamma^{m+k}},\quad k,m\geq1,	
\end{equation}
satisfies \beq\label{eqn:mu:symm}
{\mu}_{k,m}={\mu}_{m,k}\quad\mbox{for all }m,k\geq1.
\eeq
The bound \eqnref{eqn:GrunskyBounds} implies that
\beq \label{seq:inequality}
\sum_{k=1}^\infty| {\mu}_{m,k}|^2=\sum_{k=1}^\infty| {\mu}_{k,m}|^2\leq1\quad\mbox{for all }m\geq1.
\eeq
Using the modified Grunsky coefficients \eqnref{def:mu}, the formulas \eqnref{NP_series1} and \eqnref{NP_series2} become simpler: for $m =1,2,\cdots$,
	\begin{align}\label{NP_series3}
		&\KstarOmega[{\zeta_m}](\theta)=\frac{1}{2}\sum_{k=1}^{\infty}\mu_{k,m}\, {\zeta}_{-k}(\theta),\quad
		\KstarOmega[{\zeta}_{-m}](\theta)=\frac{1}{2}\sum_{k=1}^{\infty}\overline{\mu_{k,m}}\, \zeta_{k}(\theta),\\
\label{NP_series4}
		&\KOmega [\eta_m](\theta)=\frac{1}{2}\sum_{k=1}^{\infty}\mu_{m,k}\eta_{-k}(\theta),
		\quad
		\KOmega [\eta_{-m}](\theta)=\frac{1}{2}\sum_{k=1}^{\infty}\overline{\mu_{m,k}}\eta_{k}(\theta).
	\end{align}
It follows directly that $\KstarOmega$ is self-adjoint on $\Kzeta(\partial\Omega)$ (and $\KOmega$ on $\Keta(\partial\Omega)$) thanks to \eqnref{eqn:mu:symm}.

We may identify each $\varphi=\sum_{m\in\mathbb{Z}} a_m\zeta_m\in \Kzeta(\partial\Omega)$ with $(a_m)\in l^2(\mathbb{Z})$ and the operator $\KstarOmega:\Kzeta(\p\Om)\rightarrow\Kzeta(\p\Om)$ with the bounded linear operator $\left[\Kcal^*_{\p\Om}\right]:l^2(\mathbb{Z})\rightarrow l^2(\mathbb{Z})$. 
Using \eqnref{seq:inequality}, it is easy to see that
\beq\label{K_Kmat}
\big\|\Kcal^*_{\p\Om}\big\|_{\Kzeta\rightarrow\Kzeta}=\big\|[\KstarOmega]\big\|_{l^2\rightarrow l^2}\leq\frac{1}{2}.
\eeq
The matrix corresponding to $\KstarOmega:\Kzeta(\partial\Omega)\rightarrow \Kzeta(\partial\Omega)$ via the basis set $\{\zeta_m\}_{m\in\mathbb{Z}}$ (or equivalently the matrix of $\left[\Kcal^*_{\p\Om}\right]:l^2(\mathbb{Z})\rightarrow l^2(\mathbb{Z})$) 
is a self-adjoint, doubly infinite matrix given by \eqnref{eqn:matrixKstar}.
In the same way, we can identify $\Kcal_{\p\Om}:\Keta(\p\Om)\rightarrow \Keta(\p\Om)$ with the operator 
$[\Kcal_{\p\Om}]:l^2(\CC)\rightarrow l^2(\CC)$. 
Hence we have the following:
\begin{equation}\label{eqn:KstarKmatequality}
[\Kcal_{\p\Om}]=[\Kcal_{\p\Om}^*].	
\end{equation}

Since $\Kcal^*_{\p\Om}$ is self-adjoint on $K^{-1/2}(\p\Om)$, the spectrum of $\Kcal_{\p\Om}^*$ on $K^{-1/2}(\p\Om)$ lies in $[-1/2,1/2]$ from \eqnref{K_Kmat}. For a $C^{1,\alpha}$ domain, it holds that $H^{-1/2}(\p\Om)=K^{-1/2}(\p\Om)$ and, hence, the spectrum of $\Kcal^*_{\p\Om}$ on $\mathcal{H}^*$ and $\Kcal^*_{\p\Om}$ on $\Kzeta(\partial\Omega)$ coincide. 

Therefore, the result is in accordance with the fact that the spectrum of $\Kcal_{\p\Om}^*$ on $\mathcal{H}^*$ lies in $(- 1/2 , 1/2)$ \cite{Kellogg:1929:FPT}.

\smallskip

\noindent{\textbf{Proof of Theorem \ref{thm:series}.}}
First, we compute $\Scal_{\p\Om}[\tzeta_0]$. We set $z=\Psi(w)\in\CC\setminus\overline{\Om}$ and use \eqnref{eqn:Faberdefinition} and \eqnref{eqn:log_decomp} to derive
 \begin{align*}
\Scal_{\p \Om}[\tzeta_0](z)&=\frac{1}{2\pi}\int_{\p \Om}\ln|z-\tilde{z}|	\frac{1}{h(\rho,\tilde{\theta})}d\sigma(\tilde{z})\\
&=\mbox{Re}\left\{\frac{1}{2\pi}\int_0^{2\pi}\log(\Psi(w)-\Psi(\gamma e^{i\ttheta}))\right\}d\ttheta\\
&=\lim_{r\rightarrow\gamma^+}\mbox{Re}\left\{\frac{1}{2\pi}\int_0^{2\pi}\log(\Psi(w)-\Psi(r e^{i\ttheta}))d\ttheta\right\}\\
&=\lim_{r\rightarrow\gamma^+}\mbox{Re}\{\log w\}=\ln |w|.
\end{align*}
Indeed, we have $\frac{1}{2\pi}\int_0^{2\pi}F_n(\Psi(re^{i\ttheta}))d\ttheta=0$ for $r>\gamma$, $n\in\NN$  because the series in \eqnref{eqn:Faberdefinition} has a zero constant term. 
From the continuity of the single layer potential \eqnref{Scal_zeta0} follows.
By applying the jump relations \eqref{eqn:Kstarjump} to \eqnref{Scal_zeta0} we obtain
\beq\label{Kcalzero}\Kcal^*_{\p\Om}[\zeta_0]=\frac{1}{2}\zeta_0.
\eeq

Second, we expand the single layer potential on $\p\Om$ by the geometric basis $\{\zeta_{\pm m}\}_{m\in\NN}$. We use the fact that for $\varphi\in H^{-1/2}_0(\p\Om)$, the function $u:=\Scal_{\p\Om}[\varphi]$ is the unique solution to the transmission problem
\begin{equation}\label{eqn:SLtransmission}
	\ds\begin{cases}
	\ds\Delta u=0\quad&\text{in } \Rtwo\setminus\partial\Omega,\\
		\ds	u\big|^+-u\big|^-=0\quad&\text{a.e. on }\partial \Omega,\\[1mm]
	\ds	\frac{\partial u}{\partial \nu}\Bigr|^{+}-\frac{\partial u}{\partial \nu}\Bigr|^-=\varphi\quad&\text{a.e. on }\partial\Omega,\\
	\ds	u(x)=O({|x|^{-1}})\quad &\text{as } |x|\rightarrow\infty.
	\end{cases}
\end{equation}
If we set $u$ as
$$u(z)=		\begin{cases}
			\ds F_m(z)\quad&\text{for }z\in\Omega,\\[2mm]
			\ds F_m(z)-w^m+\gamma^{2m}\overline{w^{-m}}\quad &\text{for } z\in\CC\setminus\overline{\Om},
		\end{cases}
$$
then it satisfies \eqnref{eqn:SLtransmission} with
$$
\varphi(z)=\frac{\partial u}{\partial \nu}\Bigr|_{+}-\frac{\partial u}{\partial \nu}\Bigr|_-
=\frac{\partial}{\partial\nu}\left(-w^m+\gamma^{2m}\overline{w^{-m}}\right)\Big|_+=-2m\gamma^m \tzeta_m(\theta)\quad\mbox{a.e. on }\p\Om.$$
Indeed, the above equation holds for $z\in\p\Om$ which is not a corner point (see Lemma \ref{lemma:regularity} for differentiability).
Therefore, for each $m=1,2,\dots$ it holds that
	\beq\label{SL:simple}
				\SingleOmega[\tzeta_m](z)=
		\begin{cases}
			\ds-\frac{1}{2m\gamma^m}F_m(z)\quad&\text{for }z\in\Omega,\\[3mm]
			\ds-\frac{1}{2m\gamma^m}\left(F_m(z)-w^m+\gamma^{2m}\overline{w^{-m}}\right)\quad &\text{for } z\in\CC\setminus\overline{\Om}.
		\end{cases}
		\eeq
We remind the reader that for each $m\in\NN$, the Faber polynomial satisfies
\beq\label{Fm_exterior}
F_m(z)=w^m+c_{m,1}w^{-1}+c_{m,2}w^{-2}+\cdots\quad \mbox{with }z=\Psi(w)\in\CC\setminus\overline{\Om}.
\eeq
Equation \eqnref{eqn:seriesSLpositive} follows from  \eqnref{Fm_exterior}. In view of the conjugate property $$\SingleOmega[\tzeta_{-m}](z)=\Scal_{\p\Om}[\overline{\tzeta_m}](z)=\overline{\SingleOmega[\tzeta_{m}](z)},$$ we complete the proof of (a).


Now, we consider the double layer potential on $\p\Om$. One can easily show that for any $\psi\in L^2(\p\Om)$, the function $v:=\DoubleOmega\psi$ is the unique solution to the following problem:
\begin{equation}\label{eqn:DLtransmission}
	\ds\begin{cases}
	\ds\Delta v=0\quad&\text{in } \Rtwo\setminus\partial\Omega,\\
		\ds	v\big|^+-v\big|^-=\psi\quad&\text{a.e. on }\partial \Omega,\\[1mm]
	\ds	\frac{\partial v}{\partial \nu}\Bigr|^{+}-\frac{\partial v}{\partial \nu}\Bigr|^-=0\quad&\text{a.e. on }\partial\Omega,\\
	\ds	v(x)=O({|x|^{-1}})\quad &\text{as } |x|\rightarrow\infty.
	\end{cases}
\end{equation}
It is straightforward to see \eqnref{eqn:seriesDLzero}.
For $m=1,2,\dots$, one can observe from \eqnref{Fm_exterior} that
	\begin{align}\label{DL:simple}
			\ds	\DoubleOmega [\teta_m](z)&=
		\begin{cases}
			\ds\frac{1}{2\gamma^m}F_m(z)\quad&\text{for }z\in\Omega,\\[3mm]
			\ds\frac{1}{2\gamma^m}\left(F_m(z)-w^m-\gamma^{2m}\overline{w^{-m}}\right)\quad &\text{for } z\in\CC\setminus\overline{\Om}.
		\end{cases}
	\end{align}
From the conjugate relation $$\DoubleOmega [\teta_{-m}](z)=\overline{\DoubleOmega [\teta_{m}](z)},$$
 we complete the proof of (b).

To prove (c), we use the jump relation
\begin{align}\notag
\KstarOmega[\zeta_m]&=
	\frac{\partial}{\partial\nu}\SingleOmega[\zeta_m]\Big|^{-}+\frac{1}{2}\zeta_m\\
	&=-\frac{1}{2\sqrt{m}\gamma^m}\frac{\partial F_m}{\partial\nu}\Big|^-	+\frac{1}{2}\zeta_m
	=-\frac{1}{2\sqrt{m}\gamma^m}\frac{\partial F_m}{\partial\nu}	+\frac{1}{2}\zeta_m.\label{eqn:proofc_0}
\end{align}
For any non-corner point, say $x=\Psi(\rho_0,\widetilde{\theta})$, $h(\rho_0,\theta)$ is bounded away from zero and infinity in a neighborhood of $(\rho_0,\widetilde{\theta})$ and, hence, it holds for a sufficiently smooth function $u$ that
\begin{align}\label{eqn:eqlim}
	\frac{\partial u}{\partial \nu}\Big|^+(x)
	=\frac{1}{h(\rho_0,\widetilde{\theta})}\lim_{\rho\rightarrow \rho_0}
		\frac{\partial (u\circ\Psi)}{\partial\rho}(e^{\rho+i\widetilde{\theta}}).
\end{align}

First, we show that $\KstarOmega[\zeta_m]\in K^{-1/2}(\partial\Omega)$. 
Since $F_m(z)$ is a polynomial and $h(\rho,\theta)=|\frac{\partial \Psi}{\partial\rho}|$,
for any $\rho_1>\rho_0$ there is a constant $M>0$ such that 
\beq\label{eqn:uniform:bdd}
\left| \frac{\partial (F_m\circ\Psi)}{\partial\rho}(e^{\rho+i\theta})\right|
\leq Mh(\rho,\theta)\quad\text{for } \rho_0\leq\rho\leq\rho_1.
\eeq

From \eqnref{eqn:proofc_0} we have
\begin{align}
	\frac{1}{2\pi}\int_{\partial\Omega}
	\KstarOmega[\zeta_m] \eta_{-n} d\sigma
	&=-\frac{1}{2\sqrt{m}\gamma^m}
	\frac{1}{2\pi}\int_{\partial\Omega}\frac{\partial F_m}{\partial\nu} \eta_{-n} d\sigma
	+\frac{1}{2}\delta_{m,-n}.\label{eqn:proofc}
\end{align}
Fix $m,n\geq 1.$ Applying \eqref{eqn:eqlim} to $F_m$, it follows that
\begin{align*}
	\frac{1}{2\pi}\int_{\partial\Omega}\frac{\partial F_m }{\partial\nu}\eta_{-n}d\sigma
	&=\frac{1}{2\pi}\int_{0}^{2\pi}\lim_{\rho\rightarrow\rho_0}\frac{\partial (F_m\circ\Psi)}{\partial\rho}(e^{\rho+i\theta})\frac{e^{in\theta}}{\sqrt{n}}\,d\theta.
\end{align*}
From Lemma \ref{lemma:regularity}, $h(\rho_0,\theta)$ is integrable and $\int_{0}^{2\pi}h(\rho,\theta)d\theta$ converges to $\int_{0}^{2\pi}h(\rho_0,\theta)d\theta$ as $\rho\rightarrow\rho_0$. In view of \eqref{eqn:uniform:bdd}, we can exchange the order of the limit and the integration in the above equation by the dominated convergence theorem. 
We obtain 
\begin{align*}
		\frac{1}{2\pi}\int_{\partial\Omega}
		\frac{\partial F_m }{\partial\nu}\eta_{-n}
		d\sigma
	&=\lim_{\rho\rightarrow\rho_0}\frac{1}{2\pi}\int_{0}^{2\pi}
		\left[me^{m(\rho+i\theta)}-\sum_{k=1}^\infty kc_{m,k}e^{-k(\rho+i\theta)}\right]
	\frac{e^{in\theta}}{\sqrt{n}}d\theta\\
	&={\sqrt{m}\gamma^m}\delta_{m,-n}-\frac{\sqrt{n}c_{m,n}}{\gamma^k}.
\end{align*}
From \eqnref{eqn:proofc}, we deduce
\begin{align*}
\frac{1}{2\pi}\int_{\partial\Omega} \Kcal^*_{\partial\Omega}[\zeta_m]\eta_{-n} d\sigma
&=\frac{1}{2}\frac{\sqrt{n}}{\sqrt{m}}\frac{c_{m,n}}{\gamma^{m+n}}
\end{align*}
and $\frac{1}{2\pi}\int_{\partial\Omega} \Kcal^*_{\partial\Omega}[\zeta_m]\eta_{n} d\sigma=0$. We conclude $\KstarOmega[\zeta_m]\in K^{-1/2}(\partial\Omega)$ by the bound \eqref{eqn:GrunskyBounds}. 
By taking the complex conjugate, we can prove the formula for the negative indices. 	
We can similarly prove (c) for $\Kcal_{\p\Om}$ by using (b). 
\qed

\subsection{An ellipse case}
Let us derive the series expansions for the boundary integrals for a simple example.
Consider the conformal mapping $$\Psi(w)=w+\frac{a}{w}.$$ Then for each $\rho>\rho_0$, $\Psi(e^{\rho+i\theta})$ is a parametric representation of an ellipse.
Substituting $\Psi(w)$ into \eqref{eqn:Fabergenerating} gives
\begin{align*}
	\frac{w\Psi'(w)}{\Psi(w)-z}&=1+\frac{zw-2a}{w^2-zw+a}\\&=1+\frac{(w_1+w_2)w-2w_1w_2}{(w-w_1)(w-w_2)}=1+\left(\frac{w_1}{w-w_1}+\frac{w_2}{w-w_2}\right)\\
	&=1+\sum_{n=1}^{\infty}(w_1^n+w_2^n)w^{-n}
\end{align*}
where
\begin{equation*}
	w_1=\frac{z+\sqrt{z^2-4a}}{2}~\text{ and }~w_2=\frac{z-\sqrt{z^2-4a}}{2}.
\end{equation*}
Comparing with the right hand side of equation \eqref{eqn:Fabergenerating}, the Faber polynomials associated with the ellipse are
\begin{align*}
	F_0(z)&=1\\
	F_m(z)&=\frac{1}{2^m}\left[\left(z+\sqrt{z^2-4a}\right)^m+\left(z-\sqrt{z^2-4a}\right)^m\right],\quad m=1,2,\cdots.
\end{align*}

For each $m\in\NN$, $F_m(\Psi(w))=w^m+\frac{a^m}{w^m}$ so that the Grunsky coefficients are
\begin{equation*}
	 c_{m,k}=\begin{cases}
		a^k&\text{if }k=m,\\
		0&\text{otherwise}.
	\end{cases}
\end{equation*}
From Theorem \ref{thm:series} (c) it follows that
\begin{align}
	\ds&	\KstarOmega[\tzeta_m](z)=\frac{1}{2}\frac{a^m}{\gamma^{2m}}\tzeta_{-m}(z),\\
	\ds	&\KstarOmega[\tzeta_{-m}](z)=\frac{1}{2}\frac{\bar{a}^m}{\gamma^{2m}}\tzeta_{m}(z).
\end{align}
Hence, $\KstarOmega$ corresponds to the $2\times 2$ matrix 
\begin{equation*}
\frac{1}{2\gamma^{2m}}
	\begin{bmatrix}
		0 & {a^m}\\
		{\bar{a}^m}	& 0
	\end{bmatrix}
\end{equation*}
in the space spanned by $\tzeta_{-m}$ and $\tzeta_{m}$. 
In particular, $\KstarOmega$ has the eigenvalues and the corresponding eigenfunctions
 $$\pm \frac{1}{2}\frac{|a|^m}{\gamma^{2m}},\quad \pm\left(\frac{{a}}{|a|}\right)^m\tzeta_{-m}+\tzeta_{m},\quad m=1,2,\dots.$$

\subsection{Integral formula for the conformal mapping coefficients }\label{sec:transmission_solution}

\begin{theorem}\label{thm:a_k}
We assume the same regularity for $\Om$ as in Theorem \ref{thm:series}. 
Then, the coefficients of the exterior conformal mapping $\Psi(w)$ satisfy 
\begin{align}
\ds&\gamma^2 = \frac{1}{2\pi}\int_{\p\Om}z\overline{\varphi(z)}\, d\sigma(z),\\
\ds&a_m = \frac{\gamma^{m-1}}{2\pi}\int_{\p\Om}z|\varphi(z)|^{-m+1}(\varphi(z))^m\, d\sigma(z),\quad m=0,1,\dots,
\end{align}
where $$\varphi(z)=(I-2\Kcal^*_{\p\Om})^{-1}(\nu_1+i\nu_2)$$
and $\nu=(\nu_1,\nu_2)$ is the outward unit normal vector of $\p \Om$.
\end{theorem}

\begin{proof}
As before, we let $z=\Psi(w)$ be the exterior conformal mapping given as \eqnref{conformal:Psi}.
Since $d\sigma(z)=h(\rho,\theta)d\theta$, we have
\begin{align}\notag
\int_{\p \Om}z\tzeta_m(z)\, d\sigma(z)&=\int_0^{2\pi}\left(e^{\rho_0+i\theta}+a_0+a_1e^{-\rho_0-i\theta}+\cdots\right)\frac{e^{im\theta}}{h(\rho_0,\theta)}h(\rho_0,\theta)\, d\theta\\
&=\frac{2\pi}{\gamma^m} a_m, \qquad m=-1,1,0,\dots.\label{coeff_expression}
\end{align}

We remind the reader that, by taking the interior normal derivative of the single layer potential, $\tzeta$ satisfies
\beq
(-\frac{1}{2}I+\Kcal^*_{\p\Om})\tzeta_m=-\frac{1}{2m\gamma^m}\pd{F_m}{\nu}\Big|_{\p\Om}.
\eeq
Note that
\beq
\tzeta_m = \tzeta_0^{-m+1}\tzeta_1^m,\quad \tzeta_{-m}(z)=\overline{\tzeta_m(z)}
\eeq
and
\beq
|\tzeta_1(\theta)|=\frac{1}{h(\rho,\theta)}=\tzeta_0(\theta).
\eeq
Applying these relations to \eqnref{coeff_expression}, it follows that
$$\gamma=\frac{1}{2\pi}\int_{\p\Om}z\tzeta_{-1}(z)\, d\sigma(z)=\frac{1}{2\pi}\int_{\p\Om}z\frac{1}{2\gamma}\, \overline{(\frac{1}{2}I-\Kcal^*_{\p\Om})^{-1}\pd{F_1}{\nu}}\, d\sigma(z).$$
For $k= 0,1,2,\dots$, we have
\begin{align*}
	a_m
	&=\frac{\gamma^m}{2\pi} \int_{\partial\Omega} z\tzeta_0^{-m+1}\tzeta_1^m\, d\sigma(z)\\
	&= \frac{\gamma^m}{2\pi} \int_{\partial\Omega} z\big|\tzeta_1\big |^{-m+1}\zeta_1^m\, d\sigma(z).
	\end{align*}
Owing to the fact that $F_1(z)=z-a_0$, we deduce
$$\tzeta_1(z)=\frac{1}{2\gamma}(\frac{1}{2}I-\Kcal^*_{\p\Om})^{-1}(\nu_1+i\nu_2)\Big|_{\p\Om}.$$
Therefore we complete the proof.
\end{proof}

\section{Numerical computation}
We provide the numerical scheme and examples for the transmission problem based on the series expansions of the boundary integral operators. First, we explain how to obtain the exterior conformal mapping for a given simply connected domain in section \ref{conformal:numerical}. We then provide  the numerical computation based on the finite section method in section \ref{sec:finitesection}.
\subsection{Computation of the conformal mapping for a given curve}\label{conformal:numerical}
From Theorem \ref{thm:a_k}, one can numerically compute the exterior conformal mapping for a given curve by solving
\begin{equation}\label{eqn:eqnforconformal}
	\left(\frac{1}{2}I-\Kcal^*_{\p\Om}\right)[\varphi](z)=\nu_1+i\nu_2.	
\end{equation}
It is well known that one can solve such a boundary integral equation by applying the Nystr\"{o}m discretization for $\KstarOmega$ on $\p\Om$. There, we first parametrize $\partial\Omega$, say $z(t)$, and discretize it, say $\{z_p\}_{p=1}^P$ . We then approximate the boundary integral operator $\KstarOmega[\varphi]$ as
$$\KstarOmega[\varphi](x) \approx \sum_{p=1}^P\frac{\partial}{\partial\nu_x}\Gamma(x-z_p)\varphi(z_p) w_p,\quad x\in\partial\Omega.$$ The weights $\{w_p\}_{p=1}^P$ are chosen by numerical integration methods.  To obtain the accurate solution to \eqnref{eqn:eqnforconformal}, we apply the RCIP method \cite{Helsing:2013:SIE} (see also the references therein for further details).

Figure \ref{fig:rectangle} shows the exterior conformal mapping for the rectangular domain with height $1$ and  width $6.$ To ensure accuracy, we plot the difference $|\gamma^{(k)}-\gamma^{(k-1)}|$, where $\gamma^{(k)}$ is the logarithmic capacity of the domain $\Om$ computed with $k$ subdivisions in the RCIP method.
\begin{figure}[h!]
     \begin{center}
     \begin{subfigure}{.475\textwidth}
 	  	\centering
 	  	\includegraphics[height=5.5cm, width=7cm, trim={0.2cm 0.5cm 0.2cm 0.5cm}, clip]{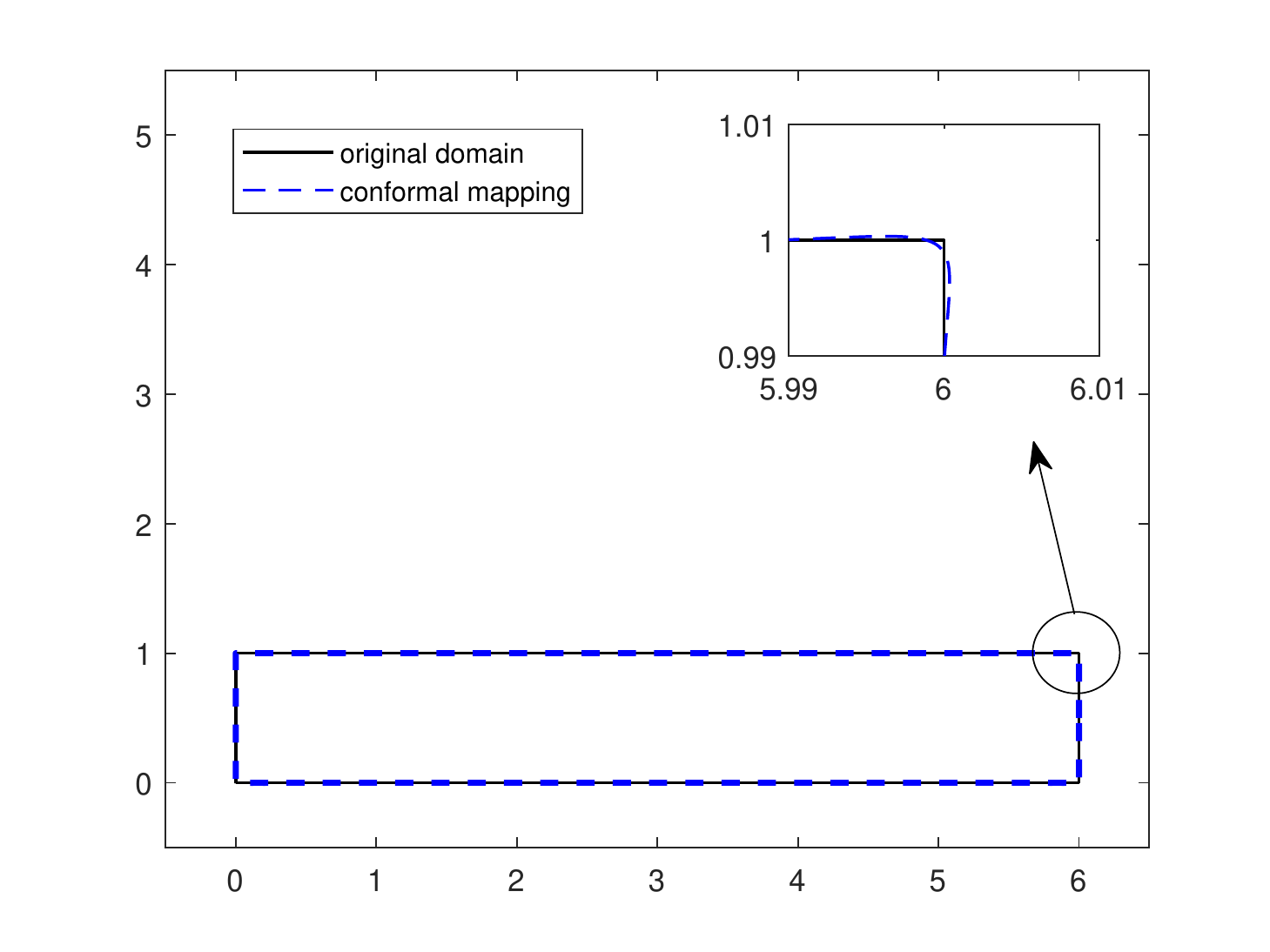}
 	  	\caption{Rectangular domain $\Omega$}
 	\end{subfigure}
 	\hfill
 	\begin{subfigure}{.475\textwidth}
 	  	\centering
 	  	\includegraphics[height=5.5cm,width=7cm,  trim={0 0.5cm 0 0.8cm}, clip]{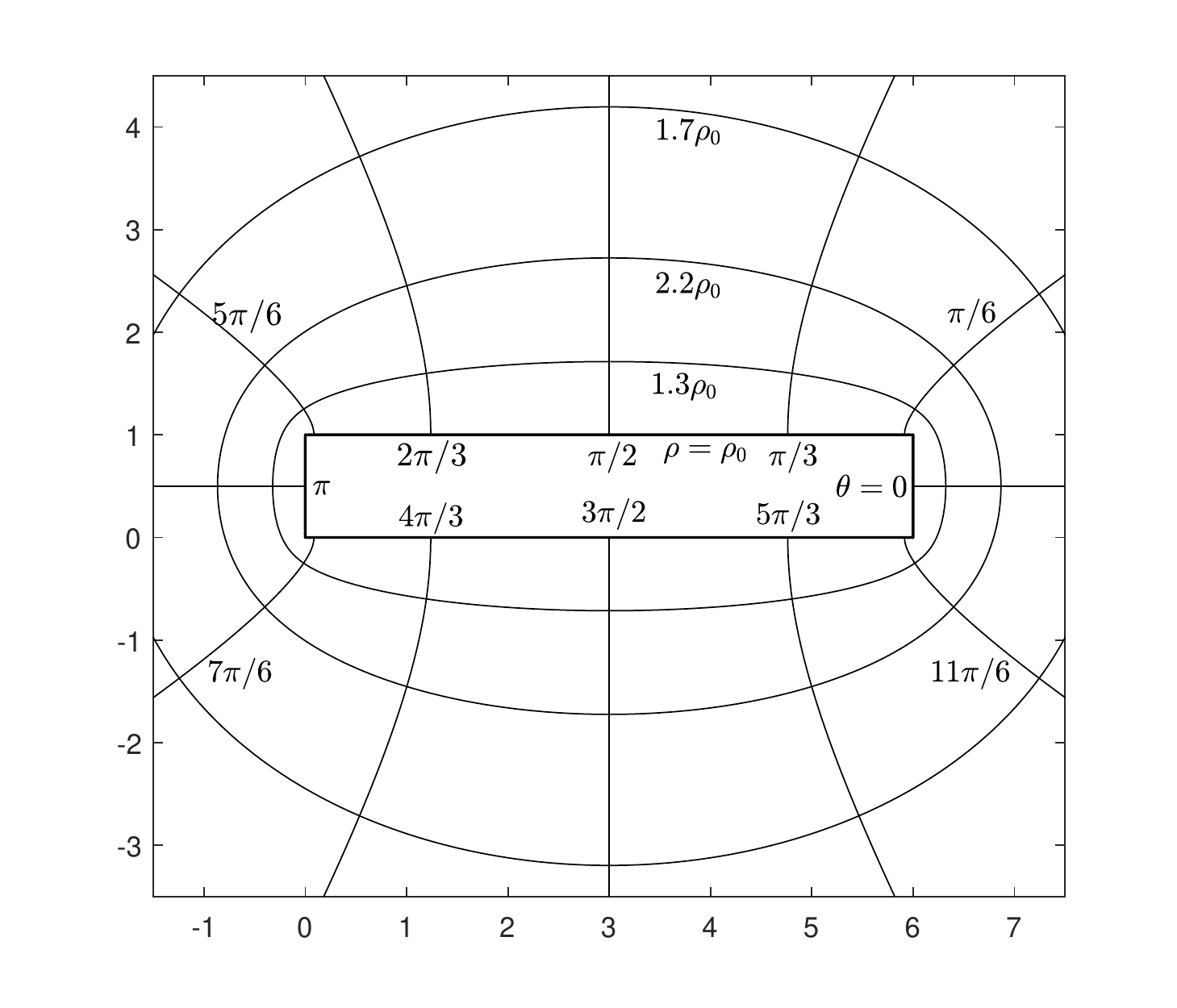}
 	  	\caption{Level coordinates curves of $\Psi(\rho,\theta)\}$}
 	\end{subfigure}
 
 \hskip 0.5cm
 	\begin{subfigure}{.475\textwidth}
 	  	\centering
 	  	\includegraphics[height=5.5cm,width=7cm,  trim={0.5cm 0 1cm 0.5cm}, clip]{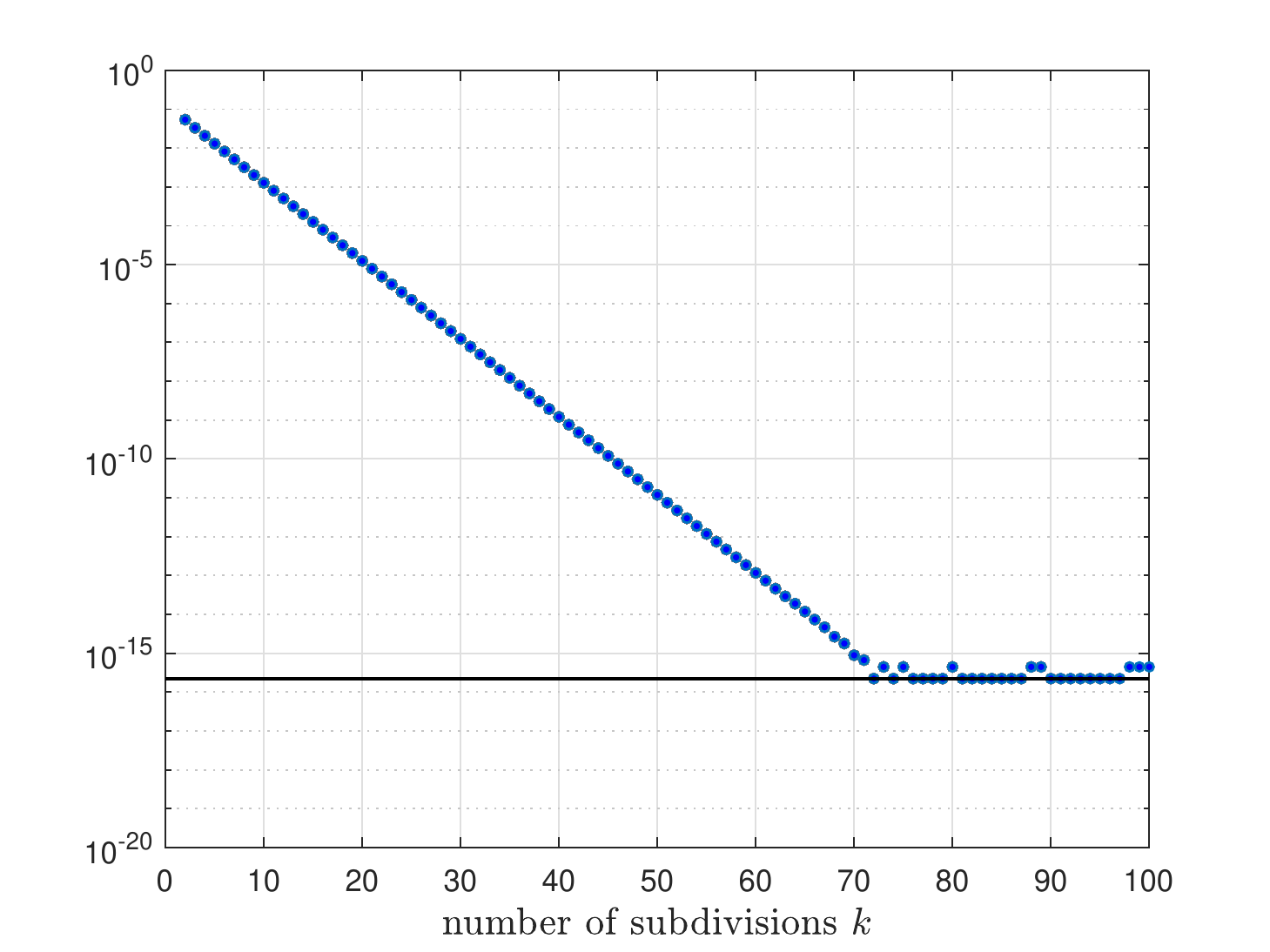}
 		 \caption{$|\gamma^{(k)}-\gamma^{{(k-1)}}|$}
 	\end{subfigure}
 	\vskip -.5cm
    \caption{(a) A rectangular-shaped domain of height $1$ and width $6$ and its close-up image near the corner. The dashed curve corresponds to the (truncated) Laurent series \eqnref{conformal:Psi} with the coefficients numerically computed as explained in section \ref{conformal:numerical}. (b) Level curves of curvilinear coordinates $(\rho,\theta)$ made by the exterior conformal mapping associated with $\Omega$. 
        (c) Convergence of $|\gamma^{(k)}-\gamma^{{(k-1)}}|$ where $k$ is the number of subdivisions for the RCIP methods. Machine precision is achieved. Logarithmic capacity $\gamma$ of the domain $\Omega$ is approximately 1.941572495732408.} 
 \label{fig:rectangle}
 \end{center}
 \end{figure}

\subsection{Numerical scheme for the transmission problem solution and the spectrum of the NP operator based on the finite section method}\label{sec:finitesection}
We now consider the numerical approximation of the solution to the boundary integral equation $(\lambda I -\mathcal{K}^*_{\p\Om})x=y$ and the spectrum of $\KstarOmega$. Once we have the infinite matrix expression for an operator, it is natural to consider its finite dimensional projection. 
More precisely speaking, we apply the finite section method to $(\lambda I -\Kcal^*_{\p\Om})x=y$ on $\Kzeta(\p\Om)$. Recall that $\Kzeta(\p\Om)$ is a separable Hilbert space. We set $H=\Kzeta$ and $$H_n=\mbox{span}\left\{\zeta_{-n},\zeta_{-n+1},\cdots,\zeta_{-1},\zeta_1,\cdots,\zeta_{n-1},\zeta_n\right\} \quad\mbox{for each }n\in\NN.$$
Then, $H_n$ is an increasing sequence of finite-dimensional subspaces of $H$ such that the union of $H_n$ is dense in $H$. 
We may identify the orthogonal projection operator to $H_n$, say $P_n$, as the operator on $l^2_0(\CC)$ given by
\begin{equation*}
	P_n(a)=(\dots,0,0,a_{-n},\dots,a_{-2},a_{-1},a_{1},a_2,\dots,a_n,0,0,\dots)\quad\mbox{for } a\in l^2_0(\CC).
\end{equation*}
Clearly, we have
$
	\|P_n a-a\|_{l^2}^2=\sum_{|m|> n}|a_m|^2\rightarrow 0\ \mbox{ as }n\rightarrow\infty,
$
so that $P_n a\rightarrow a$ as $n\rightarrow\infty$.
We denote $[\Kcal^*_{\p\Om}]_n$ the $n$-th section of $[\KstarOmega]$, that is $$[{\KstarOmega}]_n=P_n[\KstarOmega]P_n=P_n \KstarOmega P_n.$$
We identify the range of $P_n$ with $\mathbb{C}^{2n}$ and $[{\KstarOmega}]_n$ with a $2n\times 2n$-matrix, respectively.

Using the finite section of $\Kcal^*_{\p\Om}$, we can approximate the solution to the boundary integral equation and the spectrum of $\Kcal^*_{\p\Om}$ as follows:
\begin{itemize}
\item [(a)] [Computation of solution to the boundary integral equation]\\
Let $|\lambda|>\frac{1}{2}$. Then, we have $\|I-(I-\frac{1}{\lambda}\Kcal^*_{\p\Om})\|=\frac{1}{\lambda}\|\Kcal^*_{\p\Om}\|<1$. From Corollary \ref{cor:finitesection} in the appendix, the projection method for $(\lambda I-\Kcal^*_{\p\Om})$ converges, {\it i.e.},
there exists an integer $N$ such that for each $y\in\Kzeta(\p\Om)$ and $n\geq N$, there exists a solution $x_n$ in $H_n$ to the equation $P_n (\lambda I-\Kcal^*_{\p\Om})P_n x_n = P_n y$, which is unique in $H_n$, and the sequence $\{x_n\}$ converges to $(\lambda I-\KstarOmega)^{-1}y$. 
\item[(b)][Computation of the spectrum of the NP operator]\\
For self-adjoint operators on a separable complex Hilbert space, the spectrum outside the convex hull of the essential spectrum can be approximated by eigenvalues of truncation matrices.
Since $\Kcal^*_{\p\Om}$ is self-adjoint in $K^{-1/2}(\p\Om)$, the eigenvalues of the finite section operator $[{\KstarOmega}]_n$ converge to those of $\Kcal^*_{\p\Om}$ as shown in \cite[Theorem 3.1]{Bottcher:2001:AAN}. Since $[{\KstarOmega}]_n$ is a finite-dimensional matrix, one can easily compute their eigenvalues. 
\end{itemize}

\subsection{Numerical examples for the transmission problem}
We provide examples of numerical computations for the transmission problem \eqnref{cond_eqn0}. More detailed numerical results will be reported in a separate paper. 
\begin{example}
		We take $\Omega$ to be the kite-shaped domain whose boundary is parametrized by
		\begin{align*}
		\partial\Omega = \{x(t),y(t):=(\cos{t}+0.65\cos(2t)-0.65, 1.5\sin t): t\in[0,2\pi] \}.
	\end{align*}
	We set $\epsilon_c = 10000, ~\epsilon_m = 1$ and choose the entire harmonic field $H(x,y)=x.$	We computed the conformal mapping coefficients $\gamma$ and $a_k$ up to $k=50$ terms to approximate $\partial\Omega$ and truncated the matrix $[\KstarOmega]_n$ with $n=120$.
	The result is demonstrated in Figure \ref{fig:transmissionsolution}.
\end{example}

\begin{example}
We take $\Omega$ to be the boat-shaped domain whose exterior conformal mapping is given by
\begin{align*}
		{\Psi}(w):=w+\frac{0.3}{w}+\frac{0.08}{w^4}\quad\mbox{for } |w|\geq 1.
	\end{align*}
 We set $\epsilon_c = 10, ~\epsilon_m = 1$ and choose the harmonic field $H(x,y)=y$. We used the matrix $[\KstarOmega]_{n}$ with $n=120.$ The result is demonstrated in Figure \ref{fig:transmissionsolution}.
\end{example}
\begin{figure}[!h]
\includegraphics[height=6cm, width=7.7cm, trim={0.2cm 0.2cm 1cm 0.2cm}, clip]{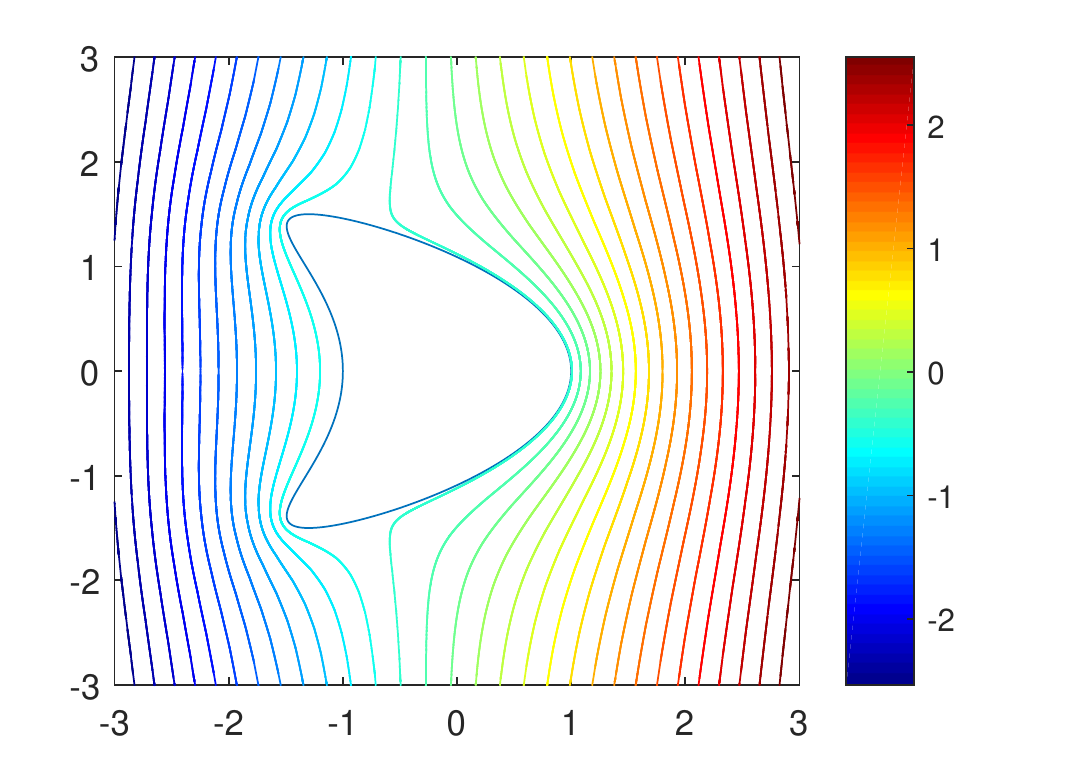}
\hskip .2cm
 \includegraphics[height=6cm, width=7.7cm, trim={0.2cm 0.2cm 1cm 0.2cm}, clip]{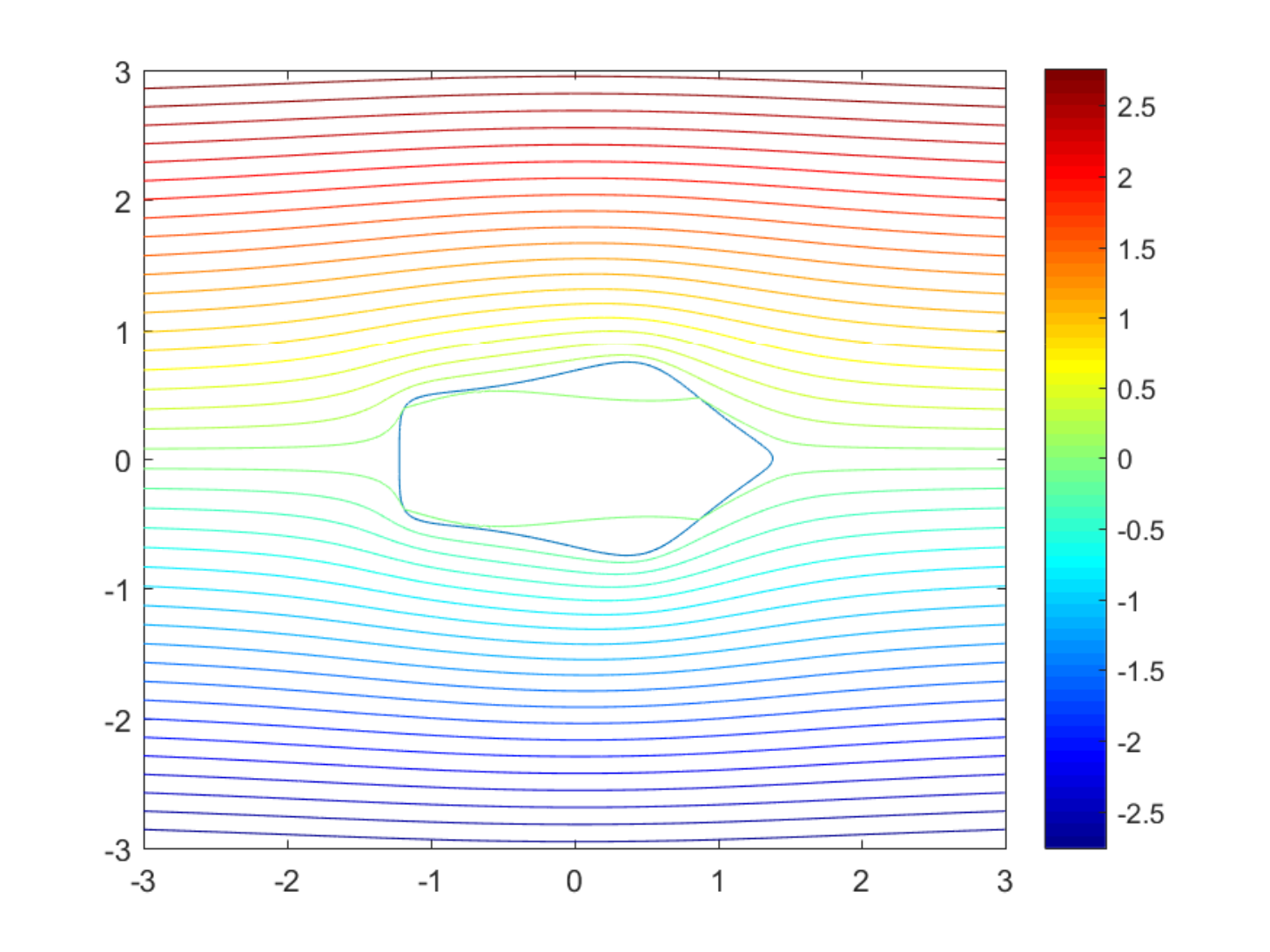}
   \caption{Level curves of the solution to the problem \eqnref{cond_eqn0} with the conditions in Example 1 (left) and Example 2 (right).
   }\label{fig:transmissionsolution}
 \end{figure}


\subsection{Numerical examples of the NP operator spectrum computation}

We give numerical examples for the NP operators of a smooth domain.
\smallskip
 \smallskip

\noindent{\textbf{Example 1.}}
Figure \ref{fig:example1} shows the eigenvalues of $\Kcal_{\p\Om}^*$ of a smooth domain $\Om$. The eigenvalues are computed by projecting the operator to $H_{100}$ space. The eigenvalues calculated using $[\Kcal^*_{\p\Om}]_N$ with various $N$ are given in Table \ref{table:example1}.
\begin{figure}[!h]
\includegraphics[height=5.5cm, width=7cm, trim={0.7cm 0.5cm 0.7cm 0.5cm}, clip]{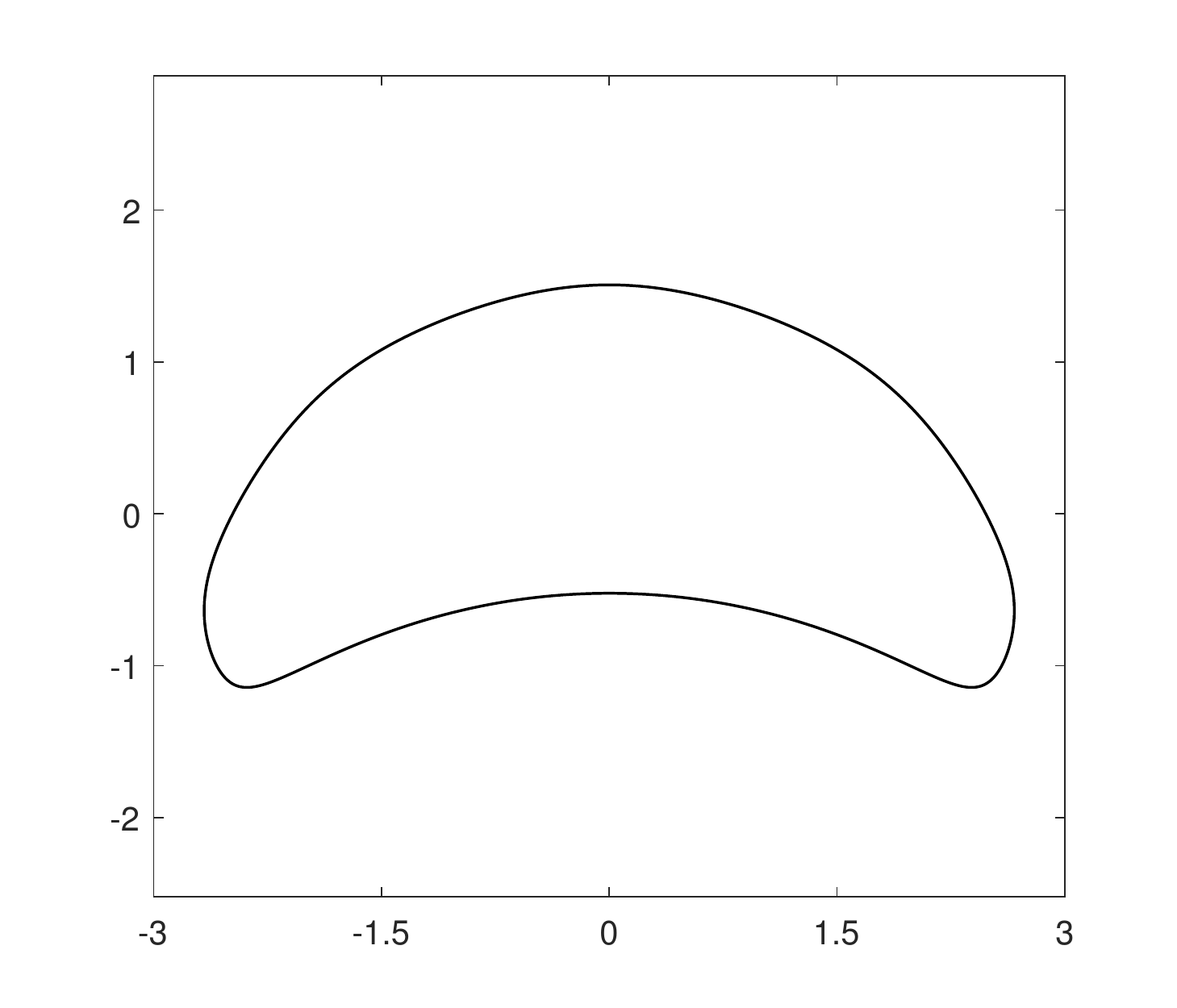}
\hskip .2cm
 \includegraphics[height=5.5cm, width=7cm, trim={0.7cm 0.5cm 0.7cm 0.5cm}, clip]{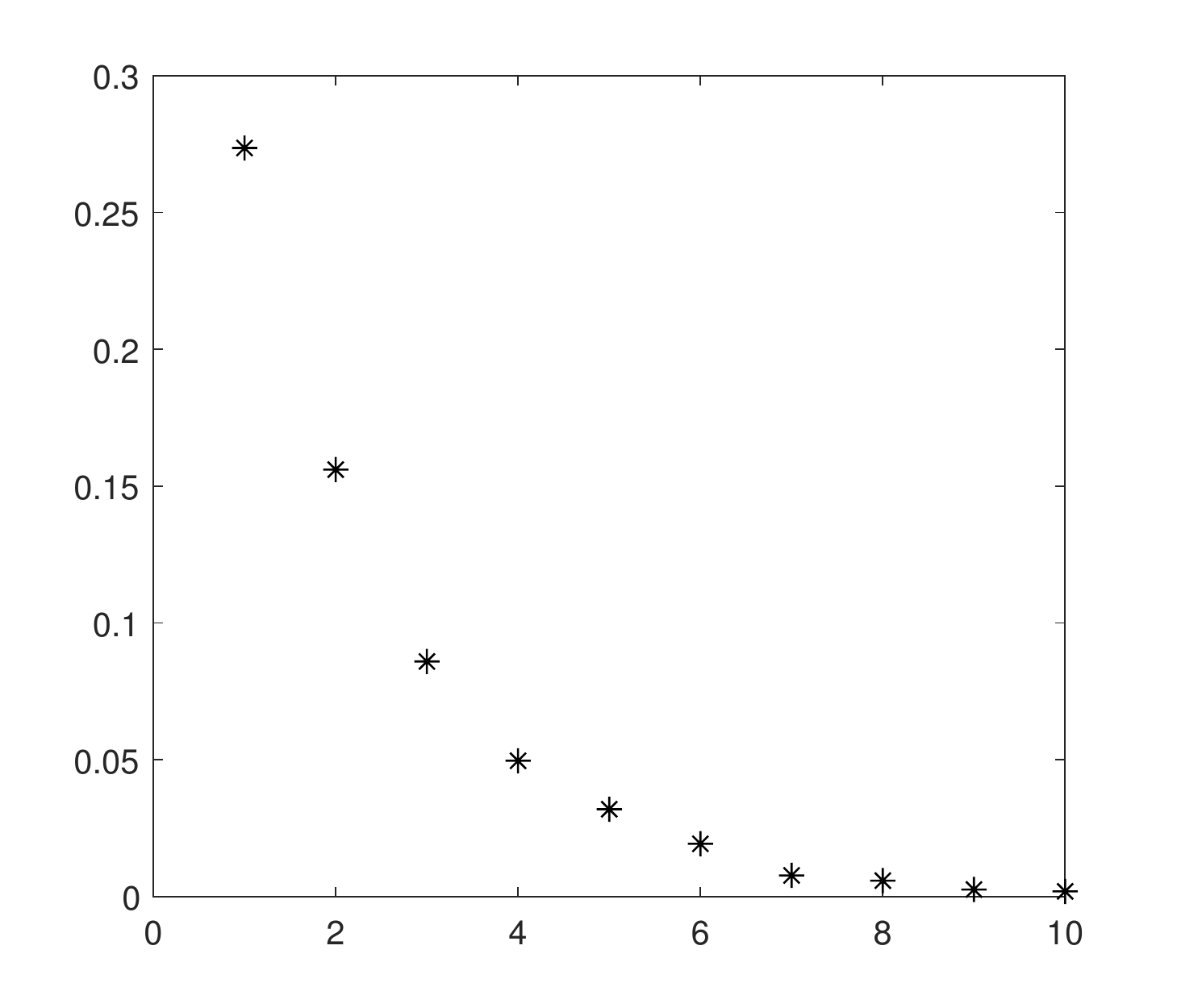}
   \caption{Eigenvalues of a smooth domain $\Om$. The domain is given by the conformal mapping $\Psi(z)=z+\frac{1.5}{z}-\frac{1.5 i}{z^2}-\frac{1.5}{z^3}+\frac{1.5 i}{z^4}+\frac{1.5}{z^5}-\frac{1.5i}{z^6}$ with $\gamma=2$. The left figure shows the geometry of $\Om$, and the right figure is the graph of the eigenvalues $\lambda_n^+$ values against $n$. The values are computed using $[\KstarOmega]_{100}$.
   }\label{fig:example1}
 \end{figure}
\begin{table}[!htbp]
\begin{center}
    \begin{tabular}{| c | c | c | c |}
    \hline
   &  $[\KstarOmega]_{10}$ & $[\KstarOmega]_{25}$ & $[\KstarOmega]_{100}$ \\ 
    \hline\hline
    $\lambda_1^+$& 0.273558129190339 & 0.273558172995812 &  0.273558172996823\\ 
    \hline
    $\lambda_2^+$& 0.156056092355289  & 0.156056664309330 & 0.156056664318575 \\ 
    \hline
    $\lambda_3^+$& 0.0859247988176952  & 0.0859480356185409 & 0.0859480356609182\\
    \hline
    $\lambda_4^+$& 0.0494465179065547  & 0.0496590062381220 & 0.0496590063829059 \\
    \hline
    $\lambda_5^+$& 0.0309734543945057  & 0.0319544507137494 & 0.0319544514943216\\
    \hline
    $\lambda_6^+$& 0.0131127593266966  & 0.0193300891854449 & 0.0193300897106055\\
    \hline
    $\lambda_7^+$& 0.00194937208308878  & 0.00776576627656122 & 0.00776578032772436 \\
    \hline
    $\lambda_8^+$& 0.000974556412340587  & 0.00585359278858573 & 0.00585361352314610\\
    \hline
    $\lambda_9^+$& 0.000178694695157472  & 0.00262352270034986 & 0.00262389455625349\\
    \hline
    $\lambda_{10}^+$& 0.000118463567811146  &0.00194010061907249 & 0.00194031822065751\\
    \hline    \end{tabular}
     \caption{Eigenvalues of $[\Kcal_{\p\Om}^*]_N$, $N=10,25,100$, for $\Om$ given in Figure \ref{fig:example1}.}\label{table:example1}
\end{center}
\end{table}

\section{Conclusion}

We defined the density basis functions whose layer potentials have exact representation in terms of the Faber polynomials and the Grunsky coefficients. These density basis functions give rise to the two Hilbert spaces $K^{\pm1/2}$ which are equivalent to the trace spaces $H^{\pm1/2}$ when the boundary is smooth.  
On these spaces the Neumann-Poincar\'{e} operators are identical to doubly infinite, self-adjoint matrix operators.
Our result provides a new symmetrization scheme for the Neumann-Poincar\'{e} operators different from Plemelj's symmetrization principle.
We emphasize that $\Kcal_{\p\Om}$ and $\Kcal_{\p\Om}^*$ are actually identical to the same matrix and, furthermore, the matrix formulation gives us a simple method of eigenvalue computation.
Since our approach requires the exterior conformal mapping coefficients to be known, we derived a simple integral expression for the exterior conformal mapping coefficients. Numerical results show successful computation of conformal mapping and eigenvalues of the Neumann-Poincar\'{e} operators. 
The present work provides a novel framework for the conductivity transmission problem.


\begin{appendices}
\section{Boundary behavior of conformal maps}\label{appen:boundarybehavior}
In this section we review regularity results on the interior conformal mappings that are provided in \cite{Pommerenke:1992:BBC}. We then derive the regularity for the exterior conformal mapping; see Lemma \ref{lemma:regularity}.

We say that a Jordan curve $C$ is of class $\mathcal{C}^m$ if it has a parametrization $w(t),0\leq t\leq 2\pi,$ that is $m$-times continuously differentiable and satisfies $w'(t)\neq0$ for all $t$. It is of class $\mathcal{C}^{m,\alpha}\;(0<\alpha<1) $ if it furthermore satisfies
\begin{equation*}
	\left|w^{(m)}(t_1)-w^{(m)}(t_2)\right|\leq M|t_1-t_2|^{\alpha}\quad\mbox{for  }t_1,t_2\in[0,2\pi].
\end{equation*}
\begin{theorem}{\rm(Kellogg-Warschawski theorem \cite[Theorem 3.6]{Pommerenke:1992:BBC})} 
	Let $f$ map $\mathbb{D}$ conformally onto the inner domain of the Jordan curve $C$ of class $\mathcal{C}^{m,\alpha}$ where $m\in\NN$ and $0<\alpha<1$. Then $f^{(m)}$ has a continuous extension to $\overline{\mathbb{D}}$ and the extension satisfies
\begin{equation*}
		\left|f^{(m)}(z_1)-f^{(m)}(z_2)\right|\leq M|z_1-z_2|^\alpha~\text{ for } z_1,z_2\in\overline{D}.
	\end{equation*}
\end{theorem}

To state the regularity results on the conformal mapping associated with a domain with corners we need the concept of {\it Dini-contiuity}:
\begin{itemize}
\item
For a function $\phi:[0,2\pi]\rightarrow\CC$ we define the modulus of continuity as
	\begin{equation*}
	\omega(\delta)=\sup\left\{\left|\phi(z_1)-\phi(z_2)\right|~:~ |z_1-z_2|\leq \delta,\ z_1,z_2\in [0,2\pi]\right\},\quad\delta> 0.
	\end{equation*}
The function $\phi$ is called {\it Dini-continuous} if
$
	\int_{0}^{\pi}\frac{\omega(t)}{t}dt<\infty
$.
The end-point of the integration interval $\pi$ could be replaced by any positive number. 
\item
We say that a Jordan curve $C$ is {\it Dini-smooth} if it has a parametrization $w(t),0\leq t\leq 2\pi,$ such that $w'(t)$ is Dini-continuous and $w'(t)\neq 0$ for all $t$. Every $C^{1,\alpha}$ Jordan curve is Dini-smooth.
\end{itemize}

Let $\Om$ be a simply connected domain whose boundary is a Jordan curve. We also let a complex function $S$ map $\mathbb{D}$ conformally onto $\Om$. We allow $\Om$ to have a corner on its boundary.
We say that $\partial \Omega$ has a corner of opening $\pi\beta\;(0\leq\beta\leq 2)$ at $S(\zeta)\in\p\Om,\;\zeta=e^{i\theta},$ if
\begin{equation*}
	\arg\left|S(e^{i t})-S(e^{i \theta})\right|\rightarrow 
	\begin{cases}
	\ds	\eta & \text{ as } t\rightarrow \theta+,\\
	\ds	\eta+\pi\beta & \text{ as } t\rightarrow \theta-.
	\end{cases}
\end{equation*}
If $\beta=1$, then at $S(\zeta)$ we have a tangent vector with direction angle $\beta$.
If $\beta$ is $0$ or $2$, then we have an outward-pointing cusp or an inward-pointing cusp, respectively.

We say that $\partial\Omega$ has a {\it Dini-smooth corner} at $S(\zeta)\in\p\Om$ if there are two closed arcs $A^{\pm}\subset \partial \mathbb{D}$ ending at $\zeta\in\partial\mathbb{D}$ and lying on opposite sides of $\zeta$ that are mapped onto Dini-smooth Jordan curves $C^+$ and $C^-$ forming the angle $\pi\beta$ at $S(\zeta)$.

\begin{theorem}{\rm(\cite[Theorem 3.9]{Pommerenke:1992:BBC})}\label{thm:regularityatboundary}
	If $\partial \Omega$ has a Dini-smooth corner of opening $\pi\beta\;(0<\beta\leq 2)$ at $S(\zeta)\neq \infty$, then the functions
	\begin{equation*}
		\frac{S(z)-S(\zeta)}{(z-\zeta)^\beta}\text{ and } \frac{S'(z)}{(z-\zeta)^{\beta-1}}
	\end{equation*}
	are continuous and $\neq 0,\infty$ in $\mathbb{\overline{D}}\cap D(\zeta,\rho)$ for some $\rho>0.$
\end{theorem}
\smallskip

\begin{lemma}[Boundary behavior of exterior conformal mapping]\label{lemma:regularity}
If $\partial \Omega$ has a Dini-smooth corner at $\Psi(w_0)$ of exterior angle $\pi\alpha \,(0<\alpha<2)$, then the functions
$$\frac{\Psi(w)-\Psi(w_0)}{(w-w_0)^\alpha}\quad\mbox{and}\quad\frac{\Psi'(w)}{(w-w_0)^{\alpha-1}}$$
are continuous and $\neq 0,\infty$ in $(\mathbb{C}\setminus\mathbb{D})\cap D(w_0;\delta)$ for some $\delta>0$, where $D(w_0;\delta)$ denotes the disk centered at $w_0$ with radius $\delta$.
\end{lemma}
\begin{proof}
We apply Theorem \ref{thm:regularityatboundary} to see the boundary behavior of $\Psi$ at the corner points.
Set $G$ to be the reflection of $\Om$ with respect to a circle centered at some point $z_0\in\Om$. Then for a conformal mapping $f$ from $\mathbb{D}$ onto $G$ we have $\Psi(w)=1/f(\frac{1}{w-z_0}+z_0)$. 
If $\p \Om$ has a corner at $\Psi(w_0)$ of exterior angle $\pi\alpha$, then $\p G$ has a corner at the corresponding point $f(z_0)$ of opening $\pi \alpha$. Since $z\rightarrow\frac{1}{z}$ is a conformal mapping from $\CC\setminus\{0\}$ onto $\CC\setminus\{0\}$, $f$ and $\Psi$ have the same regularity behavior at the corresponding corner points. From Theorem \ref{thm:regularityatboundary}, we complete the proof.
\end{proof}

\section{The Faber polynomials}\label{section:appdixFaber}
Substituting equation \eqref{conformal:Psi} into equation \eqref{eqn:Fabergenerating}, we obtain the recursion relation
\begin{equation}\label{eqn:Faberrecursion}
	-na_n=F_{n+1}(z)+\sum_{s=0}^{n}a_{s}F_{n-s}(z)-zF_n(z),\quad n\geq 0.
\end{equation}
with the initial condition $F_0(z)=1.$
The first three polynomials are
$$F_0(z)=1,\quad F_1(z)=z-a_0,\quad F_2(z)=z^2-2a_0 z+(a_0^2-2a_1).$$
Multiplying both sides of equation \eqref{eqn:Fabergenerating} by $w^{n}$ and integrating on the contour $|w|=R$, we have 
\begin{equation}\label{eqn:FabervalCauchyintegral}
	F_n(z)=\frac{1}{2\pi i}\int_{|w|=R}\frac{w^n\Psi'(w)}{\Psi(w)-z} dw,\quad z\in\overline{\Omega_r},~\gamma\leq r<R<\infty.
\end{equation}

Now let $z\in \mathbb{C}\setminus\Omega$ and consider the function
$$W(z,w) = \frac{\Psi'(w)}{\Psi(w)-z}=\frac{\Psi'(w)}{\Psi(w)-\Psi(r)},\quad r=\Psi^{-1}(z).$$
$W(z,w)$ is defined for $|w|>\gamma$ and $|r|>\gamma$ and has a only singularity at $w=r$. After considering the residue at the simple pole at $w=r$ for fixed $r$ and the simple pole at $r=w$ for fixed $w$, we see that
\begin{equation}\label{eqn:expressionPsiwr}
	\frac{\Psi'(w)}{\Psi(w)-\Psi(r)}=\frac{1}{w-r}+M(w,r)	
\end{equation}
and  $M(w,r)$ is analytic for $|w|>\gamma$ and $|r|>\gamma.$
Expanding $M(w,r)$ in double-power series and collecting the terms of the same degree with respect to $w$ we find
\begin{align}\label{eqn:expressionMwr}
	M(w,r) = a_0(r)+a_1(r)\frac{1}{w}+a_2(r)\frac{1}{w^{2}}+\cdots,\quad |w|>\gamma,~|r|>\gamma.
\end{align}
Letting $r\neq \infty$ and $w=\infty$ in the equation \eqnref{eqn:expressionPsiwr} and \eqnref{eqn:expressionMwr}
, we observe that $a_0(r)=a_1(r)=0.$ Similarly, for $r=\infty$ and $w\neq \infty$ we see that $a_k(\infty)=0,~k=2,3,\dots$. Considering the Laurent expansions of $a_k(w),~k=2,3,\dots$, one can show that the series expansion
\begin{equation}\label{eqn:Faber}
	\frac{\Psi'(w)}{\Psi(w)-\Psi(r)}-\frac{1}{w-r}=\sum_{m=1}^\infty\sum_{k=1}^\infty c_{m,k}r^{-k}w^{-m-1}
\end{equation}
holds for $|r|>\gamma$ and $|w|>\gamma$.
From \eqref{eqn:FabervalCauchyintegral}, and \eqref{eqn:Faber} we immediately observe the following relation:
\begin{equation}\label{eqn:Faberdef}
	F_m(\Psi(r))-r^m
	=\sum_{k=1}^{\infty}c_{m,k}{r^{-k}},\quad m=1,2,\dots.
\end{equation}

The Faber polynomials associated with $\Omega$ form a basis for analytic functions in $\Omega$. From the Cauchy integral formula and \eqref{eqn:Fabergenerating}, one can easily derive the following:
any complex function $f(z)$ that is analytic in the bounded domain enclosed by the curve $\{\Psi(\zeta):|\zeta|=R\}$, $R>\gamma$, admits the series expansion 
\begin{equation}\label{faberseries}
	f(z)=\sum_{m=0}^{\infty}\alpha_m F_m(z)\quad\mbox{in }\overline{\Omega}
\end{equation}
with
\begin{equation}\notag
	\alpha_m = \frac{1}{2\pi i}\int_{|w|=r}\frac{f(\Psi(w))}{w^{m+1}}\, dw,\quad \gamma<r<R.
\end{equation}

\section{Properties of the Grunsky coefficients}\label{section:Grunsky}
The Grunsky coefficients $c_{m,k}$'s can be directly computed from the coefficients of the exterior conformal mapping $\Psi$ via the recursion formula
\begin{equation}\label{eqn:cnkrecursion}
	c_{m,k+1}=c_{m+1,k}-a_{m+k}+\sum_{s=1}^{m-1}a_{m-s}c_{s,k}-\sum_{s=1}^{k-1}a_{k-s}c_{m,s},\quad m,k\geq 1,
\end{equation}
	with the initial condition $c_{n,1}=na_n$ for all $n\geq1$. We set $\sum_{s=1}^{0}=0$.
Indeed, the relation \eqref{eqn:cnkrecursion} can be easily derived by substituting \eqref{eqn:Faberdef} into \eqref{eqn:Faberrecursion} and comparing terms $r^{-k}$ of the same order. 

Applying the Cauchy integral formula on $\Omega_r$, $r>\gamma$ and \eqref{eqn:Faberdef} we derive
\begin{align*}
 	0&=\frac{1}{2\pi i}\int_{\Omega_r}F_n(z)F'_m(z)dz\\
 	&=\frac{1}{2\pi i}\int_{|w|=r}F_n(\Psi(w))F'_m(\Psi(w))\Psi'(w)dw=mc_{n,m}-nc_{m,n}.
 \end{align*} 
This identity implies the Grusnky identity
\begin{equation}
	mc_{n,m}=nc_{m,n}\quad m,n=1,2,\cdots
\end{equation}

We now review the polynomial area theorem and the Grunsky inequalities. More details can be found in \cite{Duren:1983:UF}.

The polynomial area theorem states the following:
Let $P(z)$ be an arbitrary non-constant polynomial of degree $N$ that admits the expansion
$$P(\Psi(w))=\sum_{k=-N}^{\infty}b_kw^{-k},\quad|w|>\gamma.$$
Then
	\begin{equation}\label{area}
		\sum_{k=1}^{\infty}k\left|{b_k}{\gamma^{-k}}\right|^2\leq \sum_{k=1}^{N}k\left|{b_{-k} \gamma^k} \right|^2
	\end{equation}
	with equality if and only if $\Omega$ has the measure zero.
In fact, this relation is a result of a complex form of Green's theorem: Let $f(u+iv)$ be $C^1(\overline{\Om})$, then
$$\iint_\Om \pd{f}{\bar{z}}dudv=\frac{1}{2i}\int_{\p \Om}f(z)dz,\quad z=u+iv.$$
We can easily prove the inequality \eqnref{area} by setting $f(z)=\overline{P(z)}P'(z)$:
\begin{align*}
0&\leq\int_\Om \overline{P'(z)}P'(z)dudv\\
&=\int_\Om\pd{}{\bar{z}}\left(\overline{P(z)}P'(z)\right)dudv
=\frac{1}{2i}\int_{\p\Om}\overline{P(z)}P'(z)dz\\
&=\frac{1}{2i}\int_{|w|=\gamma}\overline{P({\Psi(w)}}P'(\Psi(w))\Psi'(w)dw
=-\pi\sum_{k=-N}^\infty k|b_k|^2\gamma^{-2k}.
\end{align*}

One can derive a system of inequalities that are known as the Grunsky inequalities by applying the polynomial area theorem to $P(z)=\sum_{n=1}^{N}\frac{\lambda_n}{\gamma^n}F_n(z)$ for some complex numbers $\lambda_1,\dots,\lambda_N$; see \cite{Grunsky:1939:KSA} and \cite{Duren:1983:UF}. 

\begin{lemma}[Strong Grunsky inequalities]\label{lemma:strongGrunsky}
Let $N$ be a positive integer and $\lambda_1,\lambda_2,\dots,\lambda_N$ be complex numbers that are not all zero. Then, we have
	\begin{equation*}
		\sum_{k=1}^{\infty}k\left|\sum_{n=1}^{N}\frac{c_{n,k}}{\gamma^{n+k}}\lambda_n\right|^2\leq\sum_{n=1}^{N}n|\lambda_n|^2.
	\end{equation*}
		  Strict inequality holds unless $\Omega$ has measure zero.
	 \end{lemma}

	From the Grunsky inequalities we can derive an important bound for the Grunsky coefficients as follows. Choose some $1\leq m\leq N$ and let $\lambda_k = \delta_{mk}/\sqrt{m},~1\leq k\leq N$ in the strong Grunsky inequality. Then it holds that
\beq\label{muineq1}
\sum_{k=1}^\infty\left|\sqrt{\frac{k}{m}}\frac{c_{m,k}}{\gamma^{m+k}}\right|^2\leq 1.
\eeq

\section{Convergence of the finite section method}
We briefly introduce the convergence conditions for the finite section method. For more details we refer the reader to \cite{Gohberg:2003:BCL, Kress:2014:LIE}.

Let $H$ be a separable complex Hilbert space and $\mathcal{B}(H)$ be the linear space of bounded linear operators on $H$. We let $H_n$ be an increasing sequence of finite-dimensional subspaces of $H$ such that the union of $H_n$ is dense in $H$. We let $P_n$ be the orthogonal projection of $H$ onto $H_n$. Thus, $\|P_n\|=1$ for each $n$ and $P_n x\rightarrow x$ for every $x\in\mathcal{H}$.

For an operator $A\in \mathcal{B}(H)$ that is invertible, we say the projection method for $Ax=y$ {\it converges }if there exists an integer $N$ such that for each $y\in H$ and $n\geq N$, there exists a solution $x_n$ in $H_n$ to the equation $P_n AP_n x_n = P_n y$, which is unique in $H_n$, and the sequence $\{x_n\}$ converges to $A^{-1}y$.  

One can easily derive the following proposition and the corollary.
\begin{prop}
Let $A\in\mathcal{B}(H)$ be invertible. Then the projection method for $Ax=y$ converges if and only if there is an integer $N$ such that for $n\geq N$, the restriction of the operator $P_nAP_n$ on $H_n$ has a bounded inverse, denoted by $(P_nAP_n)^{-1}$, and
\begin{equation*}
	\sup_{n\geq{N}}\|(P_nAP_n)^{-1}\|< \infty.
\end{equation*}
\end{prop}
\begin{cor}\label{cor:finitesection}
	 If $A\in\mathcal{B}(\mathcal{H})$ is invertible and satisfies $\|I-A\|<1$, then the projection method for $A$ converges.
\end{cor}

\end{appendices}

\ifx \bblindex \undefined \def \bblindex #1{} \fi\ifx \bbljournal \undefined
  \def \bbljournal #1{{\em #1}\index{#1@{\em #1}}} \fi\ifx \bblnumber
  \undefined \def \bblnumber #1{{\bf #1}} \fi\ifx \bblvolume \undefined \def
  \bblvolume #1{{\bf #1}} \fi\ifx \noopsort \undefined \def \noopsort #1{} \fi


\begin{thebibliography}{10}

\bibitem{Ammari:2013:MSM:book}
Habib Ammari, Josselin Garnier, Wenjia Jing, Hyeonbae Kang, Mikyoung Lim, Knut
  S{\o}lna, and Han Wang.
\newblock {\em Mathematical and statistical methods for multistatic imaging},
  volume 2098.
\newblock Springer, 2013.

\bibitem{Ammari:2004:RSI:book}
Habib Ammari and Hyeonbae Kang.
\newblock {\em Reconstruction of small inhomogeneities from boundary
  measurements}, volume 1846.
\newblock Springer, 2004.

\bibitem{Ammari:2019:SRN}
Habib Ammari, Mihai Putinar, Matias Ruiz, Sanghyeon Yu, and Hai Zhang.
\newblock Shape reconstruction of nanoparticles from their associated plasmonic
  resonances.
\newblock {\em J. Math. Pures Appl.}, 122:23--48, 2019.

\bibitem{Ando:2016:APR}
Kazunori Ando and Hyeonbae Kang.
\newblock Analysis of plasmon resonance on smooth domains using spectral
  properties of the {N}eumann-{P}oincar\'{e} operator.
\newblock {\em J. Math. Anal. Appl.}, 435(1):162--178, 2016.

\bibitem{Bonnetier:2018:PRB:preprint}
Eric Bonnetier, Charles Dapogny, Faouzi Triki, and Hai Zhang.
\newblock The plasmonic resonances of a bowtie antenna.
\newblock {\em arXiv preprint arXiv:1803.02614}, 2018.

\bibitem{Bottcher:2001:AAN}
A.~B\"{o}ttcher, A.~V. Chithra, and M.~N.~N. Namboodiri.
\newblock Approximation of approximation numbers by truncation.
\newblock {\em Integr. Equat. Oper. Th.}, 39(4):387--395, 2001.

\bibitem{Caratheodory:1913:GBR}
C.~Carath\'{e}odory.
\newblock \"{U}ber die gegenseitige {B}eziehung der {R}\"{a}nder bei der
  konformen {A}bbildung des {I}nneren einer {J}ordanschen {K}urve auf einen
  {K}reis.
\newblock {\em Math. Ann.}, 73(2):305--320, 1913.

\bibitem{Ciraci:2012:PUL}
C.~Cirac{\`\i}, R.~T. Hill, J.~J. Mock, Y.~Urzhumov, A.~I.
  Fern{\'a}ndez-Dom{\'\i}nguez, S.~A. Maier, J.~B. Pendry, A.~Chilkoti, and
  D.~R. Smith.
\newblock Probing the ultimate limits of plasmonic enhancement.
\newblock {\em Science}, 337(6098):1072--1074, 2012.

\bibitem{Duren:1983:UF}
P.L. Duren.
\newblock {\em Univalent Functions}, volume 259 of {\em Grundlehren der
  mathematischen Wissenschaften}.
\newblock Springer-Verlag New York, 1983.

\bibitem{Escauriaza:1992:RTW}
L.~Escauriaza, E.~B. Fabes, and G.~Verchota.
\newblock On a regularity theorem for weak solutions to transmission problems
  with internal {L}ipschitz boundaries.
\newblock {\em Proc. Amer. Math. Soc.}, 115(4):1069--1076, 1992.

\bibitem{Escauriaza:2004:TPS}
Luis Escauriaza and Marius Mitrea.
\newblock Transmission problems and spectral theory for singular integral
  operators on {L}ipschitz domains.
\newblock {\em J. Funct. Anal.}, 216(1):141--171, 2004.

\bibitem{Escauriaza:1993:RPS}
Luis Escauriaza and Jin~Keun Seo.
\newblock Regularity properties of solutions to transmission problems.
\newblock {\em Trans. Amer. Math. Soc.}, 338(1):405--430, 1993.

\bibitem{Faber:1903:PE}
Georg Faber.
\newblock {\"U}ber polynomische {Entwickelungen}.
\newblock {\em Math. Ann.}, 57(3):389--408, 1903.

\bibitem{Fabes:1992:SRC}
Eugene Fabes, Mark Sand, and Jin~Keun Seo.
\newblock The spectral radius of the classical layer potentials on convex
  domains.
\newblock In {\em Partial differential equations with minimal smoothness and
  applications ({C}hicago, {IL}, 1990)}, volume~42 of {\em IMA Vol. Math.
  Appl.}, pages 129--137. Springer, New York, 1992.

\bibitem{Gohberg:2003:BCL}
Israel Gohberg, Seymour Goldberg, and Marinus~A. Kaashoek.
\newblock {\em Basic classes of linear operators}.
\newblock Birkh\"{a}user Verlag, Basel, 2003.

\bibitem{Grunsky:1939:KSA}
Helmut Grunsky.
\newblock Koeffizientenbedingungen f\"{u}r schlicht abbildende meromorphe
  {F}unktionen.
\newblock {\em Math. Z.}, 45(1):29--61, 1939.

\bibitem{Helsing:2013:SIE}
Johan Helsing.
\newblock Solving integral equations on piecewise smooth boundaries using the
  {RCIP} method: a tutorial.
\newblock {\em Abstr. Appl. Anal.}, ID 938167, 2013.

\bibitem{Helsing:2017:CSN}
Johan Helsing, Hyeonbae Kang, and Mikyoung Lim.
\newblock Classification of spectra of the {N}eumann-{P}oincar\'{e} operator on
  planar domains with corners by resonance.
\newblock {\em Ann. Inst. H. Poincar\'{e} Anal. Non Lin\'{e}aire},
  34(4):991--1011, 2017.

\bibitem{Helsing:2013:PCC}
Johan Helsing and Karl-Mikael Perfekt.
\newblock On the polarizability and capacitance of the cube.
\newblock {\em Appl. Comput. Harmon. Anal.}, 34(3):445--468, 2013.

\bibitem{Kang:2017:SRN}
Hyeonbae Kang, Mikyoung Lim, and Sanghyeon Yu.
\newblock Spectral resolution of the {N}eumann-{P}oincar\'{e} operator on
  intersecting disks and analysis of plasmon resonance.
\newblock {\em Arch. Ration. Mech. Anal.}, 226(1):83--115, 2017.

\bibitem{Kang:2018:SPS}
Hyeonbae Kang and Mihai Putinar.
\newblock Spectral permanence in a space with two norms.
\newblock {\em Rev. Mat. Iberoam.}, 34(2):621--635, 2018.

\bibitem{Kellogg:1929:FPT}
Oliver~Dimon Kellogg.
\newblock {\em Foundations of potential theory}.
\newblock J. Springer, 1953.

\bibitem{Kenig:1994:HAT}
Carlos~E Kenig.
\newblock {\em Harmonic analysis techniques for second order elliptic boundary
  value problems}, volume~83 of {\em CBMS}.
\newblock American Mathematical Soc., 1994.

\bibitem{Khavinson:2007:PVP}
Dmitry Khavinson, Mihai Putinar, and Harold~S. Shapiro.
\newblock Poincar\'{e}'s variational problem in potential theory.
\newblock {\em Arch. Ration. Mech. Anal.}, 185(1):143--184, 2007.

\bibitem{Krein:1998:CLO}
M.~G. Krein.
\newblock Compact linear operators on functional spaces with two norms.
\newblock {\em Integr. Equat. Oper. Th.}, 30(2):140--162, 1998.

\bibitem{Kress:2014:LIE}
Rainer Kress.
\newblock {\em Linear integral equations}, volume~82 of {\em Applied
  Mathematical Sciences}.
\newblock Springer, New York, third edition, 2014.

\bibitem{Lim:2001:SBI}
Mikyoung Lim.
\newblock Symmetry of a boundary integral operator and a characterization of a
  ball.
\newblock {\em Illinois J. Math.}, 45(2):537--543, 2001.

\bibitem{Milton:2006:CEA}
Graeme~W Milton and Nicolae-Alexandru~P Nicorovici.
\newblock On the cloaking effects associated with anomalous localized
  resonance.
\newblock {\em P. Roy. Soc. A-Math. Phy.}, 462(2074):3027--3059, 2006.

\bibitem{Pendry:2006:CEF}
J.~B. Pendry, D.~Schurig, and D.~R. Smith.
\newblock Controlling electromagnetic fields.
\newblock {\em Science}, 312(5781):1780--1782, 2006.

\bibitem{Perfekt:2017:ESN}
Karl-Mikael Perfekt and Mihai Putinar.
\newblock The essential spectrum of the {N}eumann-{P}oincar\'{e} operator on a
  domain with corners.
\newblock {\em Arch. Ration. Mech. Anal.}, 223(2):1019--1033, 2017.

\bibitem{Pommerenke:1992:BBC}
Christian Pommerenke.
\newblock {\em Boundary behaviour of conformal maps}, volume 299 of {\em
  Grundlehren der mathematischen Wissenschaften}.
\newblock Springer-Verlag, Berlin, Germany, 1992.

\bibitem{Schiffer:1981:FEG}
Menahem Schiffer.
\newblock Fredholm eigenvalues and {G}runsky matrices.
\newblock {\em Ann. Polon. Math.}, 39:149--164, 1981.

\bibitem{Smirnov:1968:FCV}
Vladimir~Ivanovich Smirnov and N.~A. Lebedev.
\newblock {\em {Functions of a complex variable: constructive theory}}.
\newblock M.I.T. Press, Cambridge, Massachusetts, 1968.

\bibitem{Suetin:1974:POR}
P.~K. Suetin.
\newblock {\em Polynomials orthogonal over a region and {B}ieberbach
  polynomials}.
\newblock American Mathematical Society, Providence, R.I., 1974.

\bibitem{Verchota:1984:LPR}
Gregory Verchota.
\newblock Layer potentials and regularity for the {D}irichlet problem for
  {L}aplace's equation in {L}ipschitz domains.
\newblock {\em J. Funct. Anal.}, 59(3):572--611, 1984.

\bibitem{Wala:2018:CMD}
Matt Wala and Andreas Kl\"{o}ckner.
\newblock Conformal mapping via a density correspondence for the double-layer
  potential.
\newblock {\em SIAM J. Sci. Comput.}, 40(6):A3715--A3732, 2018.

\end{thebibliography}
\end{document}